%% file: main.tex
\documentclass[11pt,final,3p,times,authoryear]{elsarticle}
\usepackage[T1]{fontenc}
\usepackage{amsmath, amsthm, amssymb}
\usepackage{graphicx}
\usepackage[colorlinks=true, urlcolor=blue, citecolor=black, linkcolor=black,hidelinks]{hyperref} 
\usepackage{subcaption}
\usepackage{cleveref}
\usepackage{diagbox}

\usepackage[authoryear]{natbib}
\usepackage{tikz}
\usetikzlibrary {arrows.meta}

\usepackage{todonotes}
\usepackage{comment}
\usepackage{float}
\usepackage{cleveref}
\usepackage{csquotes}
\usetikzlibrary{math}
\usetikzlibrary{decorations.pathreplacing}

\newtheorem{theorem}{Theorem}[section]

\DeclareMathOperator{\E}{E}
\DeclareMathOperator{\inte}{int}

\begin{document}

\begin{frontmatter}
    \title{Construction Methods for Space-Filling Heterogeneous Topological Interlocking Assemblies}

\author{Meike Weiß\textsuperscript{1} and Alice C. Niemeyer\textsuperscript{2}
\vspace{10pt}\\
Chair of Algebra and Representation Theory, RWTH Aachen University, Germany\\
\textsuperscript{1}weiss@art.rwth-aachen.de, \textsuperscript{2}alice.niemeyer@art.rwth-aachen.de
}

\begin{abstract}
    Deforming fundamental domains of wallpaper groups provides a systematic way to generate non‑convex blocks which admit topological interlocking assemblies (TIAs). We use this approach to construct TIAs that fully occupy the space between two parallel planes and incorporate multiple block types. In addition to wallpaper groups, semiregular tessellations are employed in the construction of such TIAs. These construction methods open up an extensive design space for TIAs, expanding the possibilities of feasible interlocking systems and creating new opportunities for architectural and material design. Several resulting block families can be interpreted as geometric realizations of generalized Truchet tiles or decorated lozenge tilings and, with suitable colouring rules, we establish a one‑to‑one correspondence between these tilings and specific TIAs. This framework enables a systematic investigation of symmetric and asymmetric assemblies derived from diverse block types.
\end{abstract}
\begin{keyword}
Topological Interlocking, Wallpaper Group, Semiregular Tilings, Truchet Tilings
\end{keyword}
\end{frontmatter}

\section{Introduction}
Building in a resource‑efficient manner is becoming increasingly important as sustainable construction practices gain prominence.
One well-known approach in this context is the use of modular and mortarless self-supporting designs, such as topological interlocking assemblies (TIAs).
An arrangement of blocks with a subset of blocks defined as the frame is a \emph{topological interlocking assembly} (TIA) if fixing the frame prevents every non-empty finite subset of blocks from moving.

The idea of topological interlocking assemblies can be traced back to the 18th century \citep{abeille_memoire_1735}. Its modern study has been initiated by~\cite{dyskin_new_2001,dyskin_toughening_2001} who introduced the concept as a novel material design.
The most familiar TIAs are arrangements of blocks placed between two parallel planes. Within this setting, two principal cases must be distinguished: assemblies of convex polyhedra (such as the Platonic solids) and assemblies of non-convex polyhedra. 
For convex polyhedra, neighbour blocks typically touch only along some parts of their faces, leaving gaps between adjacent blocks and failing to fill the entire space between the two parallel planes. 
For non-convex polyhedra, several construction methods exist that fully fill the space between the two parallel planes~\citep{Dyskin2003,VoroNoodles,TopologicalSymmetry,DelaunayLofts}. The advantage of these non-convex polyhedra is the easy assembly in practice, as each block intersects the parallel planes in a full face.
Among the existing construction methods, the approach of~\cite{TopologicalSymmetry}, which is based on wallpaper groups, offers a systematic way to construct TIAs with a certain symmetry.

While many existing TIAs rely on a single block type, assemblies composed of multiple distinct blocks offer a richer design space: they allow greater variability, new opportunities for architectural structures and a broader range of mechanical performance.
\cite{WEIZMANN201618} presented TIAs composed of various convex polyhedra. This raises the question of which TIAs composed of different block types can completely fill the space between two parallel planes. Examples for such non-convex blocks and their construction method are considered by \cite{Tubular}. 

In this paper, we introduce a comprehensive and systematic construction method for TIAs composed of different blocks that completely fill the space between two parallel planes. Our approach employs all feasible wallpaper groups and semiregular tessellations to generate TIAs consisting of one or more non‑convex block types. This framework enables a more systematic construction of blocks with both flat and curved interfaces than existing methods.
For each applicable wallpaper group, we provide explicit examples. In addition, we demonstrate that for certain wallpaper groups and specific deformations, the resulting assembly necessarily requires at least two different block types.
The presented framework of constructing TIAs composed of different blocks substantially extends the scope of existing constructions and demonstrates, through numerous new examples, the breadth of configurations made possible by our method based on~\cite{TopologicalSymmetry}.
Moreover, we show that certain TIAs are closely connected to Truchet tilings or decorated lozenge tilings. In particular, we interpret the blocks as geometric realisations of these tilings tailored to the interlocking assembly. By introducing a colour rule, we obtain a one-to-one correspondence between certain TIAs and the corresponding tilings.
This combinatorial framework allows us to investigate which symmetric and asymmetric assemblies of multiple block types exist.

\section{Preliminaries}

The two-dimensional \emph{Euclidean group} $\E(2)$ is the collection of all distance‑preserving transformations of the plane. It comprises translations, rotations, reflections and glide reflections of the Euclidean plane.
A discrete subgroup of the Euclidean group $\E(2)$ that contains two linearly independent translations is known as a \emph{wallpaper group}. 
Such groups capture the symmetries of planar periodic patterns, potentially involving rotations, reflections, and glide reflections. There exist exactly 17 distinct wallpaper groups, as shown by~\cite{armstrong1997groups}. 
Throughout this work, we use the crystallographic notation introduced in \cite{InternationalTablesA2002}.
A \emph{fundamental domain} of a wallpaper group is a smallest region of the plane whose images under the group's symmetries cover the entire Euclidean plane without overlapping interiors. 
In contrast, a \emph{translation cell} of a wallpaper group $G$ is the smallest region of the plane that tiles the entire plane solely through the two linearly independent translations of $G$.
In \Cref{fig:example_pattern} a pattern generated by the wallpaper group $p4$, featuring $90^\circ$ rotational symmetry, is shown together with an example of both a fundamental domain (blue) and a translation cell (red).

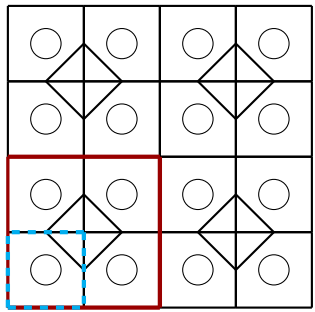
\begin{figure}[H]
    \centering
    \input{Example_pattern}
    \caption{$p4$ pattern with a fundamental domain coloured blue and a translation cell coloured red.}
    \label{fig:example_pattern}
\end{figure}

Fundamental domains can be transformed into other fundamental domains by applying the \emph{Escher Trick},  named after the Dutch artist M.C. Escher.
For this purpose, let $G$ be a wallpaper group and $F$ a fundamental domain of $G$. In general, the fundamental domain does not need to have straight edges, but we are assuming straight edges here to simplify the description.  
Observe that if the wallpaper group includes a reflection, some edge is fixed by this reflection. Any deformation of that edge would then have to be reflection‑invariant, which is not possible. Hence, the edge cannot be deformed.
First, we determine edge pairs by identifying which edges of $F$ are mapped to each other under the action of $G$. 
Depending on the choice of the fundamental domain, it may happen that more than two edges form an edge pair.
Assume $(e_1,e_2)$ to be one of the computed edge pairs. For $e_1=\{v_1,v_2\}$, we define an injective curve $\gamma_{e_1}:[0,1]\rightarrow\mathbb{R}^2$ satisfying $\gamma_{e_1}(v_i)=v_i$ for $i=1,2$. 
For example, $e_1$ can be deformed by inserting an intermediate point $p$ with a piecewise linear path from $v_1$ to $v_2$ that passes through $p$. Other piecewise linear choices are possible as well as smooth ones.
Given a deformation $\gamma_{e_1}$, the corresponding deformation of $e_2$ is uniquely prescribed by the symmetry constraints of the group $G$. Analogously, we define the deformations of the other edges where the images of all deformation curves under the action of $G$ may touch each other, but they may not intersect. By taking the orbits of the defined curves for all edge pairs, a new fundamental domain of the wallpaper group $G$ is obtained.

The left side of \Cref{fig:constr_versatile} shows a square with the vertices $v_1=(0.5,-0.5),\,v_2=(-0.5,-0.5),\,v_3=(-0.5,0.5)$ and $v_4=(0.5,0.5)$. This square acts as a fundamental domain for the wallpaper group $p4$.
To highlight the group’s symmetries, edges that are identified under the group action are shown in matching colours.
In this example, the point $p = (0,0)$ is the intermediate point describing the deformation of the edge $\{v_1, v_2\}$. The resulting piecewise linear deformation paths of the edges $\{v_1, v_2\}$ and $\{v_2, v_3\}$ are depicted as red dashed lines in \Cref{fig:constr_versatile}.
For the edge $\{v_1, v_4\}$, the point $p$ can also be used as a deformation point, since the resulting paths only touch and do not intersect.With this the deformation of $\{v_3, v_4\}$ is also determined. The described deformations transform the square into a rectangular fundamental domain. \Cref{fig:constr_versatile} visualizes this transition by showing the original square with the newly obtained rectangle.

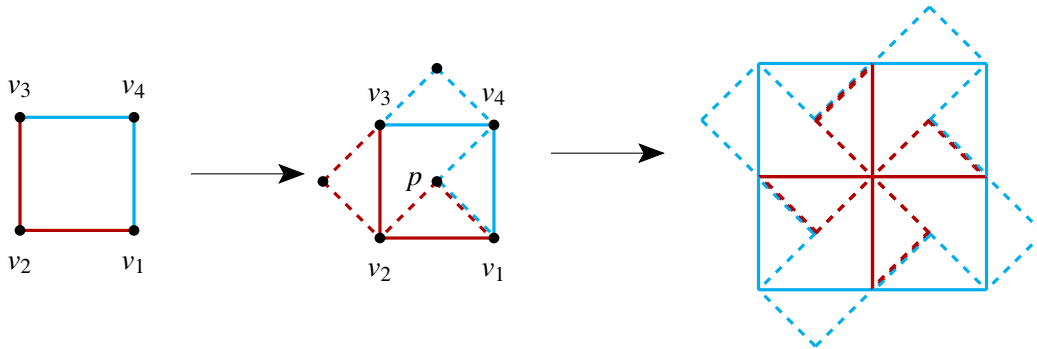
\begin{figure}[H]
    \centering
    \input{Construction_Versatile}
    \caption{Deformation of a square into a rectangle, each serving as a fundamental domain of the wallpaper group $p4$.}
    \label{fig:constr_versatile}
\end{figure}

This approach is used to construct space-filling assemblies of blocks whose arrangement follows the symmetries of the chosen wallpaper group~\citep{TopologicalSymmetry}.
Note that a family $(X_i)_{i\in I}$ of blocks, i.e.\ 3-dimensional solids, indexed by a countable set $I$, is called an \emph{assembly} if any two blocks intersect, at most, in their boundaries.
An assembly of blocks is called \emph{planar} if all blocks of the assembly lie between two parallel planes, with each block having at least one face in each plane. If the blocks touch one another in such a way that the entire space between the two parallel plane is filled, the assembly is said to be \emph{space-filling}.

As mentioned above, the Escher Trick can be applied to construct three-dimensional blocks which can be arranged according to the symmetries of the chosen wallpaper group. 
For this two fundamental domains constructed via the Escher trick for a wallpaper group $G$, are placed in two parallel planes to serve as the basis for constructing the desired blocks.
Interpolating between these two domains then yields a three‑dimensional block.
By applying the symmetries of the wallpaper group $G$, the block can be arranged to form a space-filling assembly. More details on this construction method are provided by~\cite{DissTom}.
The deformation illustrated in \Cref{fig:constr_versatile} yields the \emph{Versatile Block}, shown in \Cref{subfig:versatile} and originally introduced by~\cite{TopologicalSymmetry}. By construction, the Versatile Block can be assembled according to the symmetry of the wallpaper group $p1$, $pg$ and $p4$.

The wallpaper group together with the chosen deformation determines whether the resulting arrangement forms a topological interlocking assembly (TIA). 
A \emph{topological interlocking assembly} is an assembly of blocks $(X_i)_{i\in I}$ where a designated subset $(X_j)_{j \in J}$ (with $J \subset I$) serves as a \emph{frame} which is fixed. In such an assembly, any non‑empty finite subset of blocks is prevented from moving solely by the contact to the neighbour blocks and the frame.
The construction method outlined above offers a systematic framework for generating blocks that admit TIAs if no edge of the fundamental domain is fixed under the group action.
Consequently, the construction for blocks that admit TIAs is restricted to the groups $p1, p2, pg, p2gg, p3, p4$, and $p6$.
For instance, the $p4$ arrangement of the Versatile Block is a TIA, with the perimeter blocks defined as the frame (see \Cref{fig:InterlockingVersatile}).

\begin{figure}[H]
    \centering
    \begin{subfigure}{0.45\textwidth}
        \centering
        \includegraphics[scale=0.4]{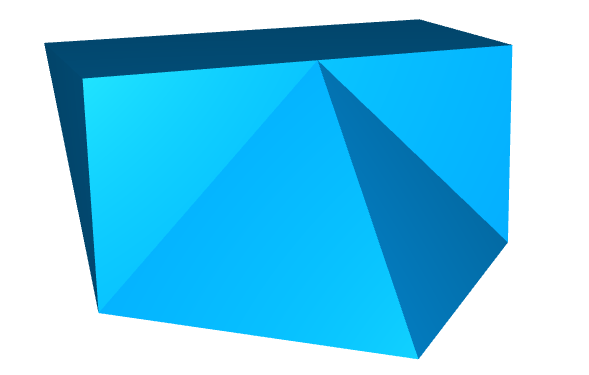}
        \caption{}
        \label{subfig:versatile}
    \end{subfigure}
    \begin{subfigure}{0.4\textwidth}
        \centering
        \includegraphics[scale=0.3]{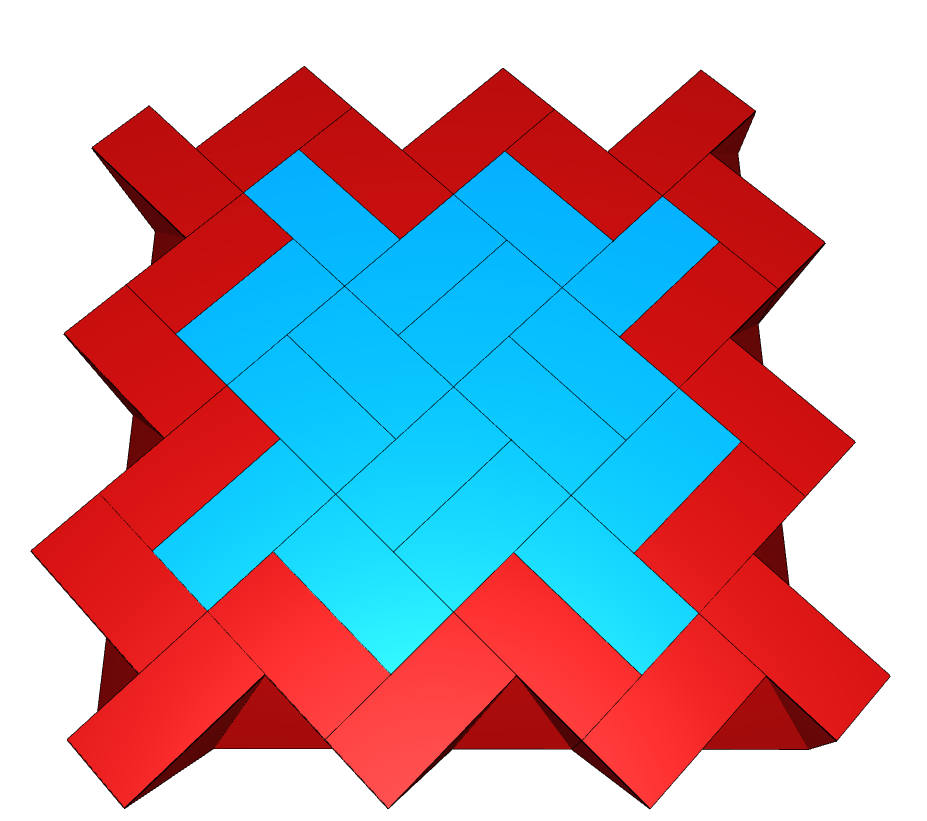}
        \caption{}
        \label{fig:InterlockingVersatile}
    \end{subfigure}
    \caption{The Versatile Block (a) and its TIA with $p4$ symmetry, shown with the frame coloured in red (b).}
    \label{fig:versatile}
\end{figure}

As shown by~\cite{bridges23}, the Versatile Block is closely connected to the concept of Truchet tilings. \cite{bridges23} establish that its geometry corresponds to a square tile divided along a diagonal into two right-angled triangles, one black and one white, referred to as a \emph{bi-triangular tile}.
As illustrated in \Cref{subfig:orient}, the orientation of the Versatile Block is determined by the underlying Truchet tile, with the black triangle corresponding to the block’s tip. Ensuring that adjacent Truchet tiles meet only at triangles with opposite colours guarantees a space-filling assembly of the Versatile Blocks. We refer to this requirement as the \emph{opposite‑colour adjacency} condition. The Truchet tiling corresponding to a $p4$ arrangement of the Versatile Block is shown in \Cref{subfig:p4Truchet}.

\begin{figure}[h!]
    \centering
    \begin{subfigure}{0.6\textwidth}
        \centering
        \hspace{-1.1cm}
        \input{bottom_left_tile.tex}
        \hspace{1.cm}
        \input{bottom_right_tile.tex}
        \hspace{1cm}
        \input{top_right_tile.tex}
        \hspace{1cm}
        \input{top_left_tile.tex}
        \vspace{0.2cm}

        \hspace{-1cm}
        \includegraphics[scale=0.14]{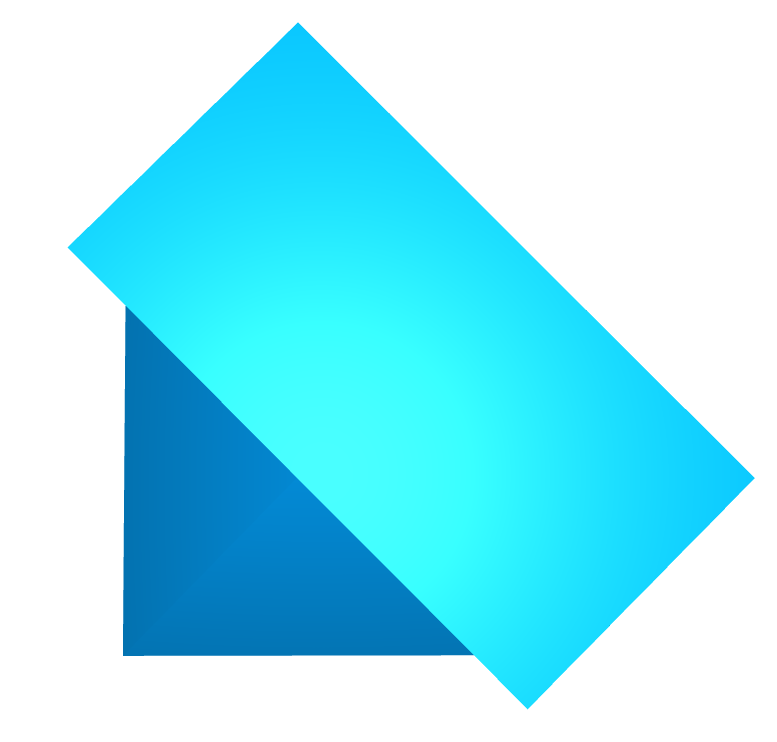}
        \includegraphics[scale=0.14]{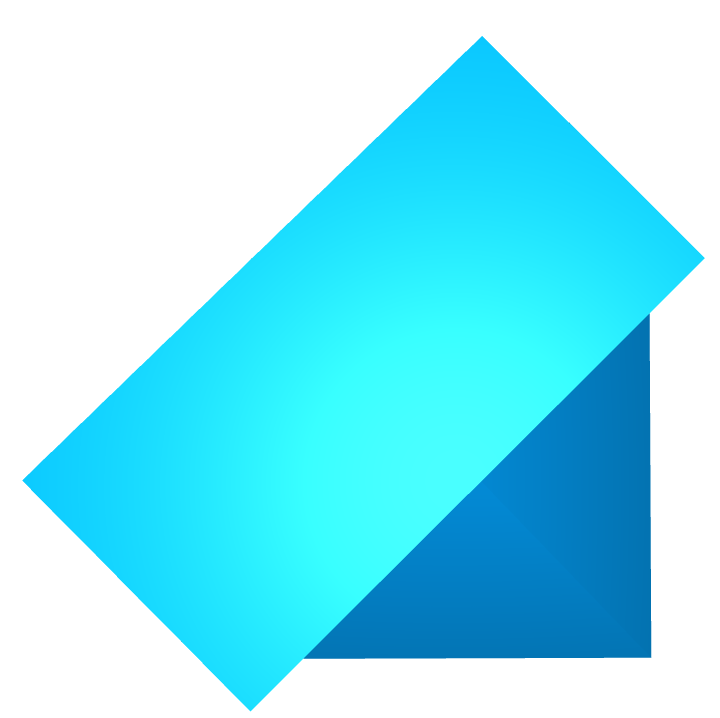}
        \hspace{0.2cm}
        \includegraphics[scale=0.14]{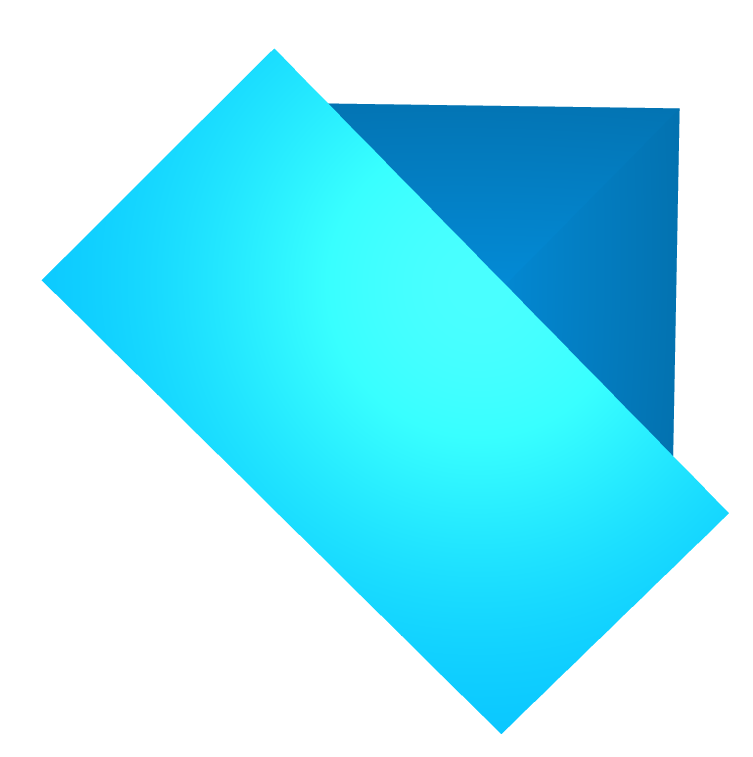}
        \hspace{0.5cm}
        \includegraphics[scale=0.14]{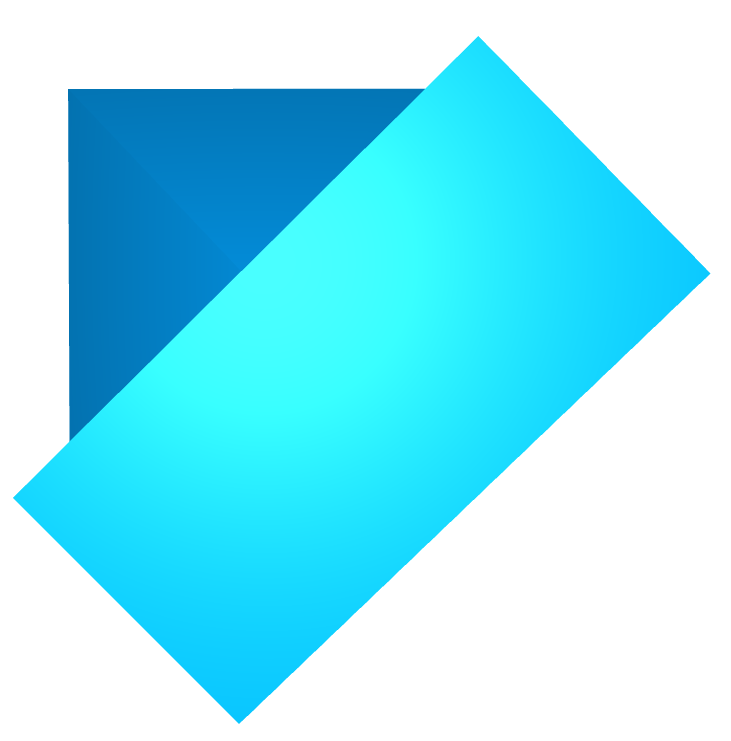}
        \caption{}
        \label{subfig:orient}
    \end{subfigure}
    \begin{subfigure}{0.3\textwidth}
        \centering
        \includegraphics[scale=0.32]{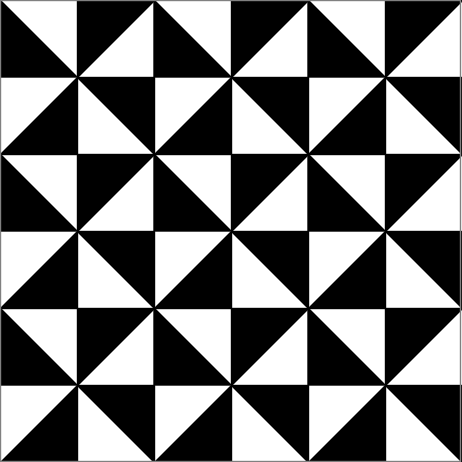}
        \caption{}    
        \label{subfig:p4Truchet}
    \end{subfigure}
    \caption{Correspondence between bi-triangular tile orientations and the Versatile Block geometry (a), and the Truchet tiling underlying the $6 \times 6$ arrangement of the Versatile Block with $p4$ symmetry (b).}
    \label{fig:versatile_ass}
\end{figure}

\section{Heterogenous planar assemblies based on the same wallpaper group}

In this section a methodology for constructing space‑filling assemblies composed of two distinct block types, which we refer to as \emph{heterogeneous} assemblies, is introduced.
Both block types are generated from the same wallpaper group via the Escher Trick and employ the same fundamental domain ensuring full geometric compatibility.

Let $G$ be a wallpaper group isomorphic to $p1, p2, pg, p2gg, p3, p4$ or $p6$, and let $F$ be a fundamental domain of $G$ with four boundary edges $e_1,\dots ,e_4$, which is possible for all groups in this list. The end vertices of the edge $e_i$ are $v_i$ and $v_{i+1}$ for $i\in\{1,\dots,4\}$ with $v_5:=v_1$. 
Each edge deformation is defined by an injective curve $\gamma_{e_i}:[0,1]\rightarrow\mathbb{R}^2$ satisfies $\gamma_{e_i}(0)=v_i$ and $\gamma_{e_i}(1)=v_{i+1}$.
The \emph{interior} of $F$ is defined as the region bounded by the edges $e_1,\dots,e_4$ without the edges itself and is denoted by $\inte(F)$.
With these definitions we can define two different types of edge deformation.
For this let $t_0=\inf\left\{t\in[0,1]\mid \gamma_{e_i}(t)\notin \{v_i+s*(v_{i+1}-v_i)\mid s\in\mathbb{R}\}\right\}$. Then $\gamma_{e_i}$ is called an \emph{inward edge deformation} if $\gamma_{e_i}(t_0)\in\inte(F)$. Analogously, $\gamma_{e_i}$ is called an \emph{outward edge deformation} if $t_0\notin\inte(F)$.
Moreover, we say that two edge deformations have the \emph{same type} if the two corresponding curves can be mapped onto each other by some isometry.

First of all, assume that $e_1$ and $e_2$ are identified by the action of $G$, i.e.\ $g(e_1)=e_2$ for some $g\in G$, as likewise $e_3$ and $e_4$. 
We define $\gamma_{e_1}$ to be an inward edge deformation which means that $\gamma_{e_2}$ is an outward edge deformation since then $e_1$ has to be mapped to $e_2$ by a rotation.
For the edge $e_3$, we choose between an inward and an outward deformation, with the deformation of $e_4$ determined by this choice.
Together with the deformation choice at $e_1$, this leads to two distinct fundamental domains.
If the chosen edge deformations of $e_1$ and $e_3$ are of the same type, one of the resulting fundamental domains exhibits rotational symmetry.
In the case where $e_1$ and $e_3$ are identified by $G$ (so $G$ contains no rotation) and likewise $e_2$ and $e_4$, the same approach yields two different fundamental domains.
However, depending on the symmetry of the chosen edge deformations and the symmetry of the group, the two resulting fundamental domains may coincide or differ only by a rotation or reflection.
When the construction yields two distinct fundamental domains, the corresponding blocks can be assembled into a heterogenous assembly. This works because both blocks employ the same pair of deformation types. We refer to one of the blocks as the \emph{obverse block} of the other.
Depending on the choice of $G$ and the symmetry of the deformation, both symmetric and asymmetric arrangements of the two blocks may occur.
If every edge pair is deformed and the outer blocks are chosen as the frame, the assembly of the resulting blocks forms a space‑filling heterogeneous topological interlocking assembly.

In this paper, we most often consider deformations of a fundamental domain where both edge pairs are deformed by the same type. Moreover, we mostly consider piecewise linear paths but all the presented constructions can be done similarly with smooth deformations.
We start by examining the $p4$ wallpaper group as a representative example. This case, together with the $p3$ case, is distinguished by the fact that it admits connections to combinatorial objects that are already well studied in the literature.

\subsection{Heterogenous assemblies based on \texorpdfstring{$p4$}{p4} symmetry}
Let $G$ be isomorphic to the wallpaper $p4$. As above, we again take the square with the vertices $v_1=(0.5,-0.5),\,v_2=(-0.5,-0.5),\,v_3=(-0.5,0.5)$ and $v_4=(0.5,0.5)$ as the fundamental domain. We label its edges by $e_1=\{v_1,v_2\},\, e_2=\{v_2,v_3\},\, e_3=\{v_3,v_4\} \text{ and } e_4=\{v_1,v_4\}$.
For the Versatile Block the edges $e_1$ and $e_4$ are deformed with the intermediate point $(0,0)$, as depicted in \Cref{fig:constr_versatile}.Thus, $e_1$ and $e_4$ are both deformed by inward edge deformations.
Suppose now that we deform $e_4$ outwards with the same deformation type, then this deformation results in a new fundamental domain for $p4$ consisting of two small squares.
The two compatible deformations of the square described above together with the used deformation curve defined by the intermediate point $p$ are illustrated in \Cref{fig:def}.

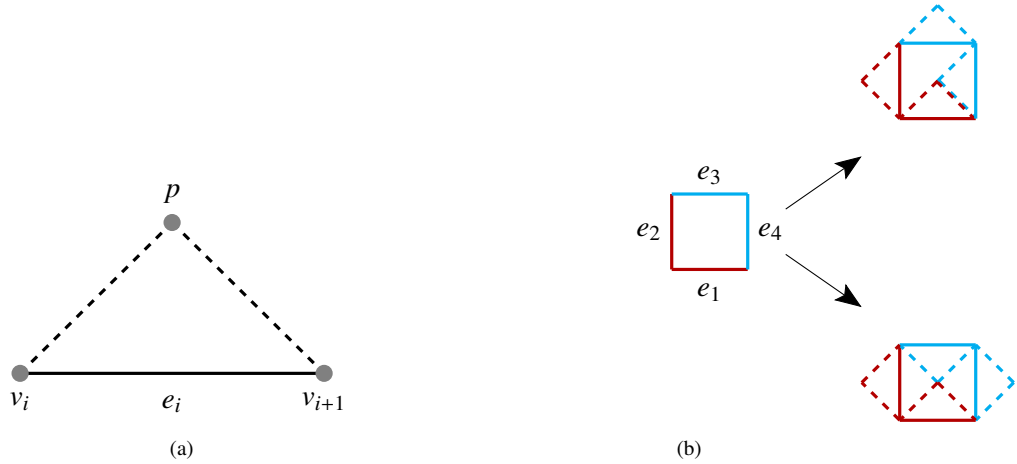
\begin{figure}[H]
    \centering
    \begin{subfigure}{0.3\textwidth}
        \centering
        \input{Deformation}
        \caption{}
        \label{subfig:deformation}
    \end{subfigure}
    \begin{subfigure}{0.5\textwidth}
        \centering
        \input{Different_Deformations}
        \caption{}
        \label{subfig:diff_deformation}
    \end{subfigure}
    \caption{Deformation of an edge (a) with the two resulting deformations (illustrated with dashed edges) of the square both respecting the $p4$ symmetry (b).}
    \label{fig:def}
\end{figure}

Placing the two squares of the deformed fundamental domain in the plane parallel to the $xy$-plane at height $z=1$ and interpolating between them yields the so‑called \emph{Bisquare Block}, shown in \Cref{fig:bisquare}. This block was originally introduced by~\cite{frezier1737}.
The Bisquare Block has 11 vertices with the coordinates
$$\left(-\frac{\sqrt{2}}{2},\frac{\sqrt{2}}{2},0\right),\left(\frac{\sqrt{2}}{2},\frac{\sqrt{2}}{2},0\right),\left(\frac{\sqrt{2}}{2},-\frac{\sqrt{2}}{2},0\right),\left(-\frac{\sqrt{2}}{2},-\frac{\sqrt{2}}{2},0\right),$$
$$\left(-\frac{\sqrt{2}}{2},\frac{\sqrt{2}}{2},1\right),\left(0,\sqrt{2},1\right),\left(\frac{\sqrt{2}}{2},\frac{\sqrt{2}}{2},1\right),
\left(\frac{\sqrt{2}}{2},-\frac{\sqrt{2}}{2},1\right),\left(0,-\sqrt{2},1\right),\left(-\frac{\sqrt{2}}{2},-\frac{\sqrt{2}}{2},1\right),\left(0,0,1\right)$$ such that the fundamental domain has an area of 2.
The triangulation of the Bisquare Block is determined by the following vertex-face incidences, where $1,\dots,11$ represent the vertices in the order above:
$$[1,2,4],[2,3,4],[5,6,7],[5,7,11],[8,9,10],[8,10,11],[4,9,10],[3,8,9],[3,4,9],[1,4,11],[1,5,11],$$
$$[4,10,11],[2,6,7],[1,5,6],[2,3,11],[2,7,11],[3,8,11],[1,2,6].$$

\begin{figure}[h!]
    \centering
    \includegraphics[scale=0.6]{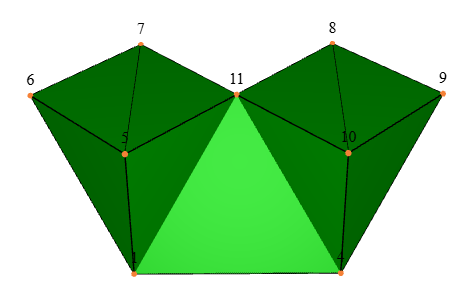}
    \includegraphics[scale=0.6]{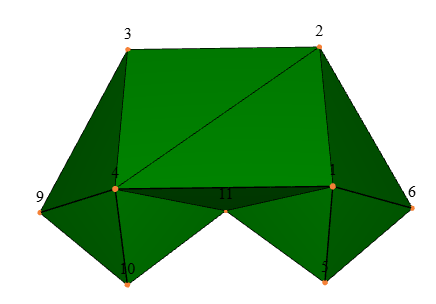}
    \caption{Two different views of the Bisquare Block.}
    \label{fig:bisquare}
\end{figure}

A special property of the Versatile and Bisquare Blocks is that their construction employs a symmetric deformation in which the intermediate point is obtained by shifting the midpoint of the edge along the direction perpendicular to the edge. 
This explains why both blocks possess at least one axis of symmetry. Note that the Bisquare Block exhibits an additional axis, as its edges were deformed symmetrically: each edge deformation is obtained from the others by a rotation.
Due to this special property, both blocks can be represented by generalized Truchet tiles. As described above, the Versatile Block corresponds to a bi‑triangular tile. Owing to its higher symmetry, the Bisquare Block can be represented by a \emph{quad‑triangular tile}, that is, a square subdivided into four triangles in which the two triangles not sharing an edge are assigned the same colour.
The quad‑triangular tile, together with the encoding of the Bisquare Block’s orientation, is shown in \Cref{fig:encoding_Bisquare}.
In contrast to the Versatile Block and the bi‑triangular tile, only two distinct orientations occur for the Bisquare Block and the quad-triangular tile, since the remaining ones coincide by symmetry.

\vspace{-0.3cm}
\begin{figure}[H]
    \centering
    \begin{subfigure}{0.5\textwidth}
        \begin{minipage}{0.3\textwidth}
            \input{top_bottom}
        \end{minipage}
        \hspace{-0.4cm}
        $\leftrightarrow$
        \hspace{0.2cm}
        \begin{minipage}{0.4\textwidth}
            \includegraphics[scale=0.25]{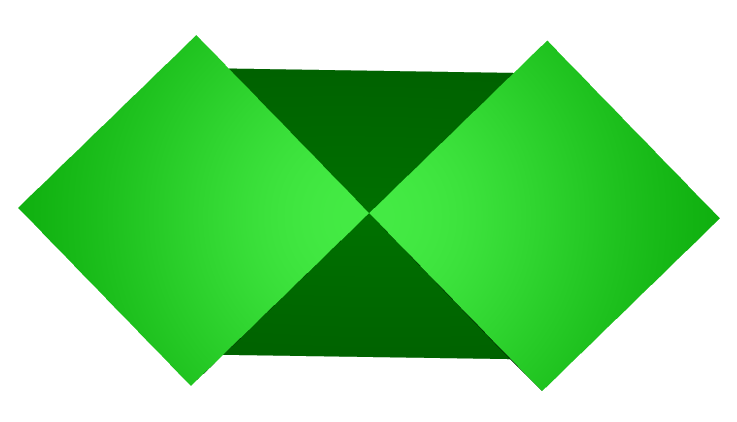}
        \end{minipage}
    \end{subfigure}
    \hspace{0.4cm}
    \begin{subfigure}{0.4\textwidth}
        \begin{minipage}{0.3\textwidth}
            \input{left_right}
        \end{minipage}
            $\leftrightarrow$
            \hspace{0.2cm}
        \begin{minipage}{0.4\textwidth}
            \includegraphics[scale=0.25]{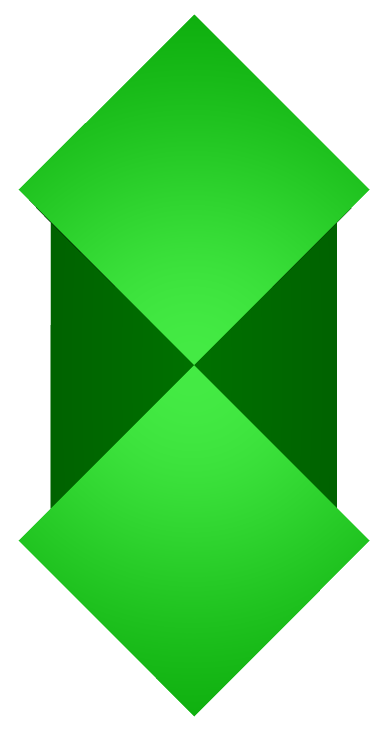}
        \end{minipage}
    \end{subfigure}
    \caption{Correspondence between the orientations of the quad-triangular tiles and the Bisquare Block.}
    \label{fig:encoding_Bisquare}
\end{figure}

A tiling composed of these quad-triangular tiles corresponds to an assembly of Bisquare Blocks if and only if the opposite-colour adjacency condition is satisfied, meaning that only triangles of different colours may touch along an edge. This leads to the following statement.

\begin{theorem}
    There exist exactly two $n\times m$ assemblies of the Bisquare Block, or equivalently, exactly two $n\times m$ square tilings composed solely of quad‑triangular tiles for $n,m\in\mathbb{N}$. These two assemblies (or tilings) exhibit $p4$ symmetry.
\end{theorem}
\begin{proof}
    First, note that the orientation of any quad-triangular tile is uniquely determined by the orientation of a single neighbouring tile and that neighbouring tiles always have opposite orientations. The last property already induces the $p4$ symmetry of the quad-triangular tiling.
    
    We begin with the top‑left square of the tiling, where the orientation of the quad‑triangular tile may be chosen freely, yielding two possible options. Fixing an orientation of the first tile determines the orientations of the quad-triangular tiles in the entire first row and also the leftmost entry of the second row. The second tile in the second row has two neighbours with a fixed orientation, and by construction they have different colours at the edges adjacent to the tile under consideration. This determines a unique feasible orientation for the tile. The same reasoning propagates across the entire tiling, ensuring that every subsequent orientation is uniquely determined. Consequently, each initial choice for the first tile produces a unique $n\times m$ tiling. Consequently, there are exactly two valid tilings with $p4$ symmetry, and the corresponding result for the arrangements of Bisquare Blocks follows immediately.
\end{proof}

One of the possible two $5\times 5$ assemblies of Bisquare Blocks together with the corresponding tiling composed of quad-triangles is depicted in \Cref{fig:ass_bisquare}.
Note that the two possible tilings composed of quad-triangles are $90^\circ$ rotations of each other, as are the corresponding two assemblies of the Bisquare Block.

\begin{figure}[H]
     \begin{subfigure}{0.45\textwidth}
        \centering
        \includegraphics[scale=0.35]{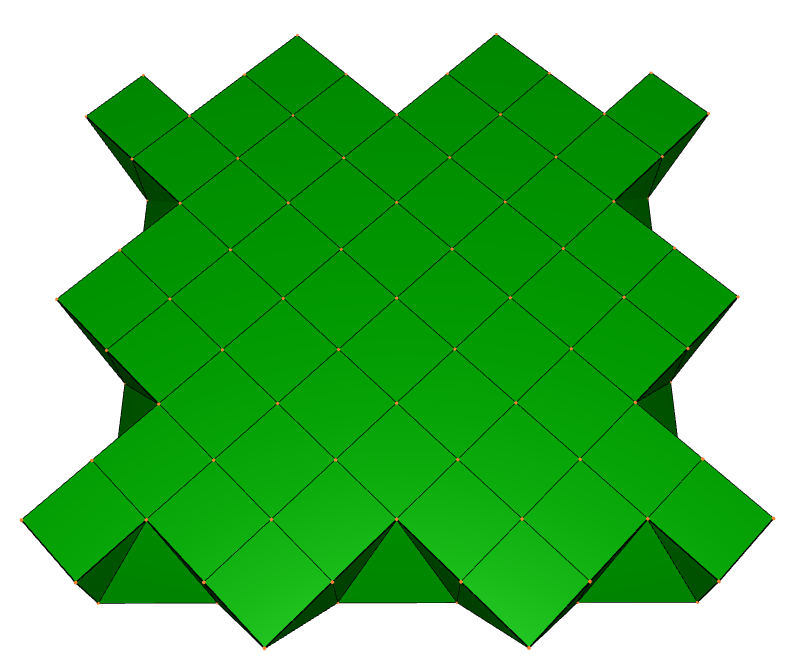}
        \caption{}    
    \end{subfigure}
    \begin{subfigure}{0.45\textwidth}
        \centering
        \includegraphics[scale=0.45]{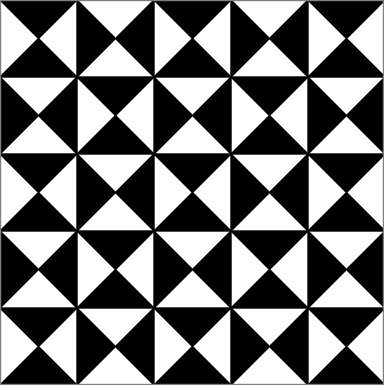}
        \caption{}    
    \end{subfigure}
    \caption{A $5\times 5$ assembly of the Bisquare Block (a) and the corresponding quad-triangular tiling (b).}
    \label{fig:ass_bisquare}
\end{figure}

Since the deformations of the Versatile and Bisquare Block are compatible, they can be arranged within a heterogenous assembly. However, the question remains which types of arrangements are actually possible. 
For this purpose, we identify possible arrangements using square tilings composed of bi- and quad-triangular tiles. Such a tiling corresponds to a valid heterogenous assembly of the Versatile and the Bisquare Block if and only if the opposite-colour adjacency condition is satisfied.
Taking the possible orientations of the tiles into account, these tilings consist of six distinct decorations of the squares.
\Cref{fig:possTilings} shows three examples of tilings composed of bi‑ and quad‑triangular tiles and the corresponding heterogenous assemblies. Besides periodic patterns, aperiodic random tilings are also possible, see \Cref{subfig:aperiodic}.

\begin{figure}[h!]
    \centering
    \hspace{0.3cm}
    \begin{subfigure}{0.3\textwidth}
        \centering
        \includegraphics[scale=0.34]{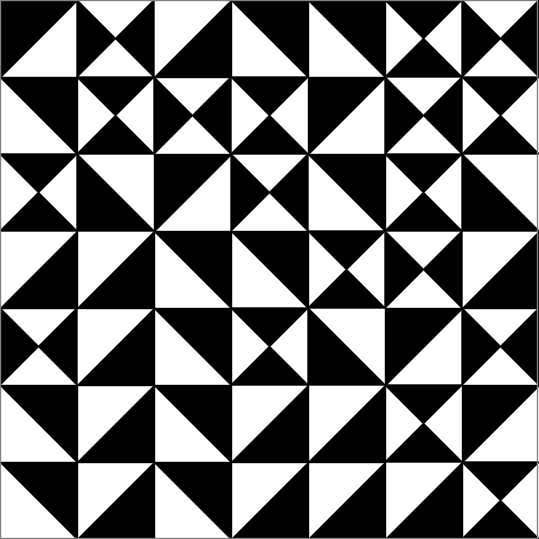}
        \caption{}
        \label{subfig:aperiodic}
    \end{subfigure}
    \hspace{0.4cm}
    \begin{subfigure}{0.3\textwidth}
        \centering
        \includegraphics[scale=0.34]{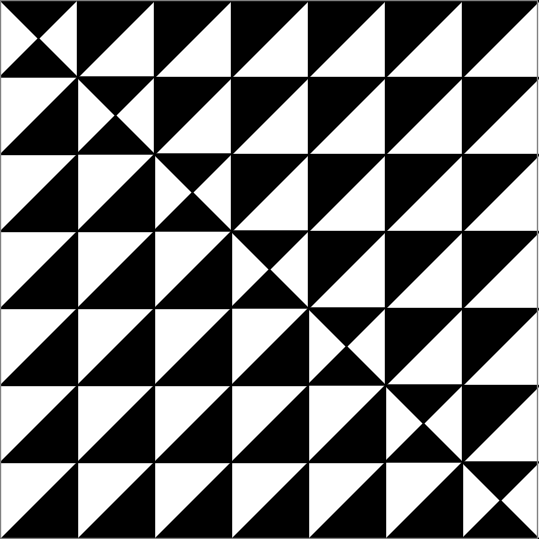}
        \caption{}    
    \end{subfigure}
    \hspace{0.5cm}
    \begin{subfigure}{0.3\textwidth}
        \centering
        \includegraphics[scale=0.34]{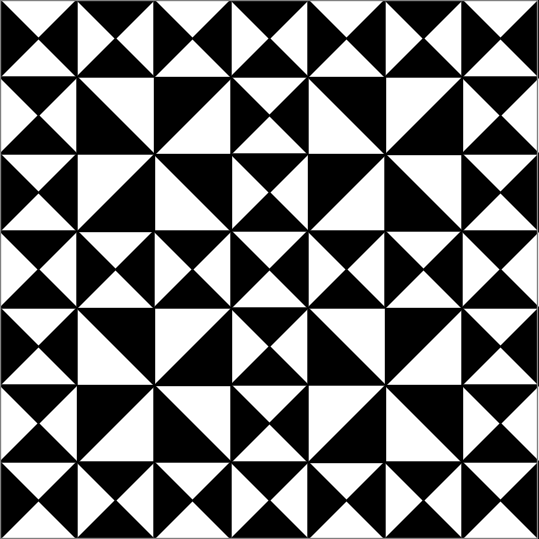}
        \caption{}    
    \end{subfigure}
    \hspace*{-0.2cm}
    \begin{subfigure}{0.3\textwidth}
        \centering
        \includegraphics[scale=0.34]{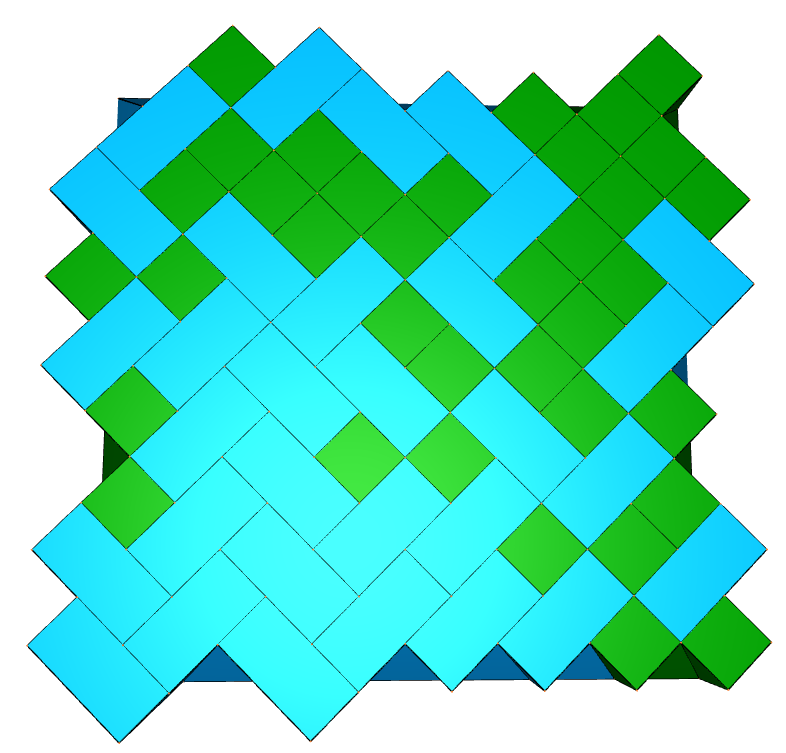}
        \caption{}    
    \end{subfigure}
    \hspace{0.7cm}
    \raisebox{10pt}{
    \begin{subfigure}{0.3\textwidth}
        \centering
        \includegraphics[scale=0.34]{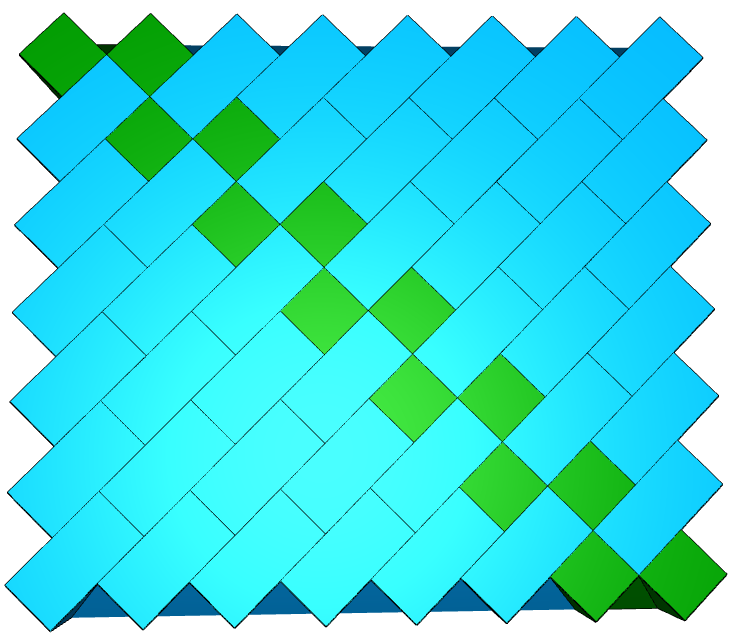}
        \caption{}    
    \end{subfigure}}
    \hspace{0.4cm}
    \raisebox{5pt}{
    \begin{subfigure}{0.3\textwidth}
        \centering
        \includegraphics[scale=0.33]{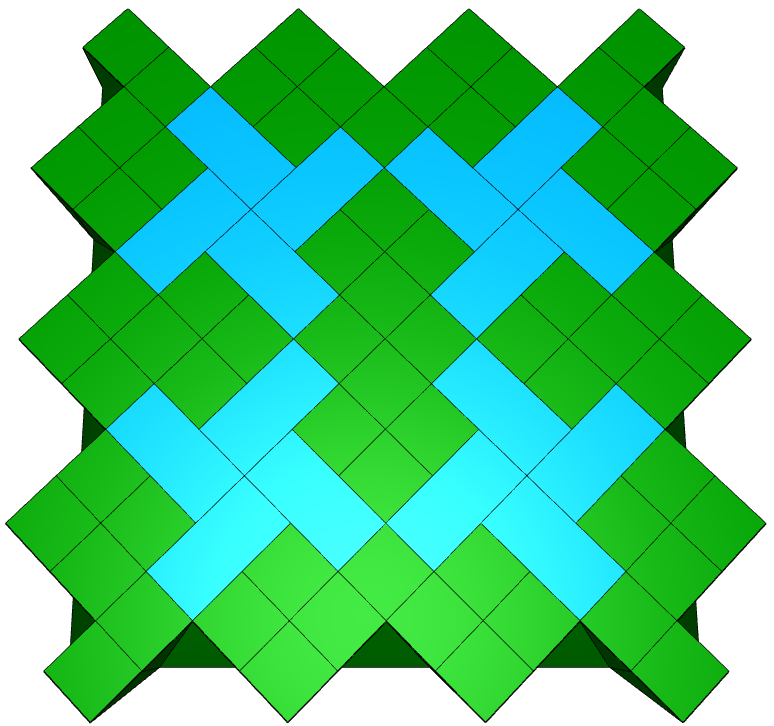}
        \caption{}    
    \end{subfigure}}
    \caption{A few selected $7\times 7$ tilings composed of bi- and quad-triangular tiles satisfying the opposite-colour adjacency condition together with the top view of their heterogenous assemblies composed of the Versatile and Bisquare Block.}
    \label{fig:possTilings}
\end{figure}

This raises the question how many $n\times m$ tilings composed of bi- and quad-triangular tiles satisfying the opposite-colour adjacency condition exist for $n,m\in\mathbb{N}$. We computed the numbers of these tilings for small $n$ and $m$ and listed them in \Cref{table:numbers}.

\begin{table}[h!]
    \centering
    \begin{tabular}[h]{|c|c|c|c|c|c|c|c|c|c|c|c|}
    \hline
    \diagbox{\textbf{$m$}}{\textbf{$n$}}  & 1 &  2 & 3 & 4 \\ 
     \hline
     1 & 6& 18& 54& 162\\
     \hline
     2 & 18 & 82 & 374 & 1706\\
     \hline
     3 & 54 & 374 & 2604& 18150\\
     \hline
     4 & 162 & 1706 & 18150 & 193662\\
     \hline
\end{tabular}
\caption{Numbers of possible $n\times m$ tilings composed of bi- and quad-triangular tiles and satisfying the opposite-colour adjacency condition.}
\label{table:numbers}
\end{table}

Surprisingly, these numbers correspond to the numbers of 3-colourings of $(n+1)\times (m+1)$-grid graphs, where the colour of the first vertex is fixed, see the sequence A078099 in the On-Line Encyclopedia of Integer Sequences~\citep{OEIS_A078099}. Note that connections between these colourings and other combinatorial objects have already been explored~\citep{Miura-ori}.
For $n,m\in \mathbb{N}$, the \emph{grid graph} $G_{n,m}$ is the graph with the vertex set $\{ 1,\dots ,n\} \times \{ 1,\dots ,m\}$, where $(i,j)$ and $(i',j')$ are adjacent if and only if $|i-i'|+|j-j'|=1$. Note that a \emph{3‑colouring} assigns one of three colours to each vertex of a graph in such a way that no two adjacent vertices share the same colour. With these definitions, it can be shown that there is indeed a correspondence between the tilings under consideration and the possible 3‑colourings of grid graphs.

\begin{theorem}
    The $n\times m$ tilings composed of bi- and quad-triangular tiles that satisfy the opposite-colour adjacency condition are in one-to-one correspondence with 3-colourings of the $(n+1)\times (m+1)$-grid graph, where the colour of the first vertex is fixed.
\end{theorem}
\begin{proof}
    The corners and edges of an $n\times m$ square tiling correspond naturally to the vertices and edges of the $(n+1)\times (m+1)$-grid graph. Conversely, the $(n+1)\times (m+1)$-grid graph can be translated to a $n\times m$ tiling by decorating each square face of the planar grid as a bi- or quad-triangular tile.
    Note that in any 3‑colouring of a grid graph, two of the non-adjacent vertices in a square face must share the same colour. The remaining two vertices may either have distinct colours or may share the same colour.

    Now, assume that $T$ is a $n\times m$ tiling composed of bi- and quad-triangular tiles satisfying the opposite-colour adjacency condition. We have to prove that this tiling corresponds to a unique 3-colouring of the $(n+1)\times (m+1)$-grid graph.
    To achieve this, we specify for each tile of $T$ a vertex colouring that induces a unique 3‑colouring of the underlying grid graph.
    We require bi‑triangular tiles to be coloured such that the vertices along their diagonal have distinct colours, while quad‑triangular tiles employ only two colours, chosen so that each diagonal connects vertices of identical colour.
    
    We first determine all possible colourings of a single tile whose top left vertex is coloured blue. There are exactly six possible decorations for a tile, and we have to show that these give exactly six distinct ways to colour the remaining three vertices of the square face.
    The two colourings where the vertices have only two colours correspond to the two orientations of the quad-triangular tile, see \Cref{subfig:e,subfig:f}. Without loss of generality we can define which colouring corresponds to which orientation, as the opposite arises by interchanging the colours green and red.
    The four colourings using three colours correspond to the four orientations of the bi-triangular tile.
    Nevertheless, it is not obvious how a given orientation of a bi-triangular tile determines the colouring, or equivalently, which side of the diagonal carries the black triangle and which the white.
    Note however, that if the colours of two of its vertices are known the other two are determined.
    Consider the top edge of the quad‑triangular tile incident to a white triangle, see \Cref{subfig:e}. 
    This edge has the colour blue for the left vertex and green for the right vertex. Thus, to ensure a consistent 3‑colouring throughout the entire assembly, there must be a bottom edge of the bi‑ and quad‑triangular tile whose vertices are coloured blue (left) and green (right) and which is incident to a black triangle, thereby determining its vertex-colouring.
    A case-by-case analysis of all edges and their colours leads to a unique correspondence between the colourings and the orientations of the bi‑triangular tiles, see \Cref{subfig:a,subfig:b,subfig:c,subfig:d} for those with a blue top left corner.
    Thus, for each decoration of the top-left tile the colouring of the four vertices is fixed as shown in \Cref{fig:3colouring}. Note that the first and second colour options are $180^{\circ }$ rotations of one another.

    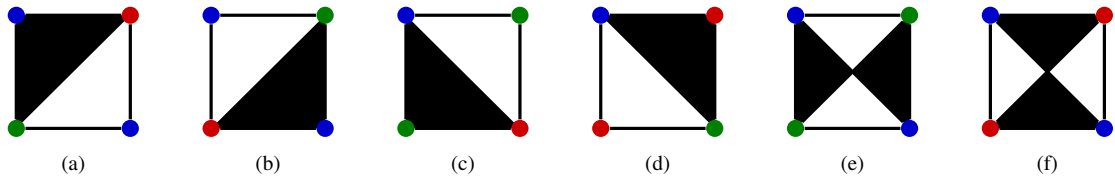
\begin{figure}[H]
        \centering
        \input{3colouring}
        \caption{The six suitable 3-colourings of the first tile depending on its decoration.}
        \label{fig:3colouring}
    \end{figure}

    With these colour assignments of the top-left tile, we now have to show that this vertex colouring can be extended to the entire grid by taking all rotational variants of the tiles into account.
    Accordingly, for each tile we verify that, once the neighbouring tiles prescribe certain vertex colours, the tile admits a unique admissible colouring.
    For this observe that every orientation of a bi‑triangular tile allows exactly three different vertex colourings, obtained by choosing one of the three colours and assigning it to the two vertices that are not connected by the diagonal, see \Cref{subfig:b,subfig:c,subfig:d} and its rotations.
    Moreover, for any edge incident to a white triangle—regardless of whether it is the top, bottom, left, or right edge—there are exactly three possible vertex‑colour combinations. The same holds for any edge incident to a black triangle.
    Thus, when colouring a bi‑triangular tile whose neighbour’s colouring is already fixed, there is exactly one possible colour choice, since the opposite‑colour adjacency condition is satisfied.
    If the colourings of two neighbouring tiles are already fixed, there is always a suitable colouring for the remaining vertex, since every possible combination of the pre‑coloured edges appears among the six options shown in \Cref{fig:3colouring} and their rotations, together with the opposite‑colour adjacency condition.
    In the case of the quad-triangular tiles, exactly two colour options exists for each orientation. Consequently, an edge that is already forced to be red and green by an adjacent tile cannot belong to a quad‑triangular tile. However, this situation cannot generally be excluded. Therefore, the colouring of the quad‑triangular tile shown in \Cref{fig:missedCol} must be included to ensure a valid 3‑colouring for any suitable tiling.

    \begin{figure}[H]
        \centering
        \input{missed_colouring}
        \caption{The third 3-colouring option for a quad-triangular tile.}
        \label{fig:missedCol}
    \end{figure}
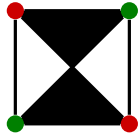
    
    With this the resulting 3-colouring is unique up to permutation of the colours and clearly induces a valid 3‑colouring of the corresponding $(n+1)\times (m+1)$-grid graph, since by construction no two adjacent vertices are coloured the same.

    Finally, assume that the $(n+1)\times (m+1)$-grid graph is given together with a 3-colouring. To obtain a tiling consisting of bi- and quad-triangular tiles, we proceed as follows: Square faces using only two colours become quad-triangular tiles, while square faces using all three colours become bi-triangular tiles, 
    with the diagonal connecting the two vertices of different colours. It remains to specify which triangles of each tile are black and which are white such that the opposite-colour adjacency condition is satisfied.
    For this, consider a horizontal edge with two incident tiles whose left vertex is blue and whose right vertex is green. We define the incident triangle above this edge to be black and the one below it to be white.
    This choice is without loss of generality, as the opposite arises by interchanging white and black.
    With this convention, and taking into account the possible colourings of the top and bottom square faces of this edge (see \Cref{fig:col_til}), we obtain that for horizontal edges coloured red–blue (left–right) or green–red (left–right), the triangle above the edge is coloured black and the one below is white. For horizontal edges coloured in the opposite direction, the assignment is reversed.
    For the vertical edges we obtain that the edges coloured blue-green, green-red and red-blue (each read from top to bottom) the triangle on the left is coloured white and the one on the right is black. For vertical edges coloured in the opposite direction, the assignment is reversed.
    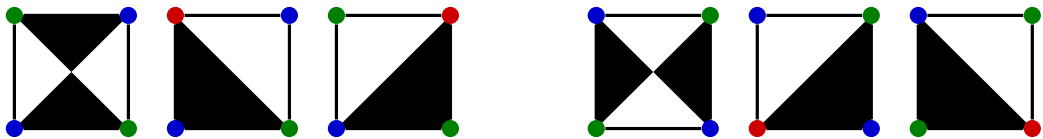
\begin{figure}[H]
        \centering
        \input{tiling}
        \caption{The possible colourings of the top and bottom square faces of a horizontal edge coloured blue-green.}
        \label{fig:col_til}
    \end{figure}
    Thus, the correspondence between the 3-colouring and the orientations of the bi- and quad-triangular tiles is uniquely defined and chosen such that the opposite-colour adjacency condition is satisfied.    
\end{proof}

In \Cref{fig:ex_correspondence}, an example of a $3\times 3$ Truchet tiling consisting of bi- and quad-triangular tiles and satisfying the opposite-colour adjacency condition, together with the unique 3-colouring of the $4\times 4$-grid graph, where the first vertex is coloured blue, is depicted.

\begin{figure}[h!]
    \centering
    \input{Example}
    \caption{An example of the one-to-one correspondence of a $3\times 3$ Truchet tiling consisting of bi- and quad-triangular tiles and satisfying the opposite-colour adjacency condition and the unique 3-colouring of the $4\times 4$-grid graph.}
    \label{fig:ex_correspondence}
\end{figure}
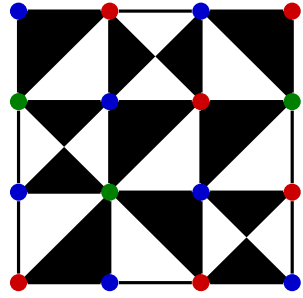

So far no closed formula for the number of 3-colourings of a grid graph is known. However, they can be computed on the basis of a matrix $M(i)$ defined recursively on $i\in\mathbb{N}$ with
$$M(1)=(1),\, M(i+1)=\begin{pmatrix}
    M(i) & M^t(i)\\
    0 & M(i)
\end{pmatrix}.$$
Then the number of possible $(n-1)\times (m-1)$ Truchet tilings consisting of bi- and quad-triangular tiles and satisfying the opposite-colour adjacency condition, and thus the number of possible arrangements of the Versatile and Bisquare Block, is given by the sum of the entries of $\left(M(m)+M(m)^t\right)^{(n-1)}$~\citep{OEIS_A078099}. Note that two tilings or arrangements may in fact be identical up to rotation or reflection.

Beyond their planar arrangements, the Versatile and Bisquare Blocks also support various spatial tessellations due to the choice of the coordinates. Non-planar configurations of the Versatile Block have been presented by~\cite{bridges23} whereas \Cref{fig:3D} illustrates possible non-planar arrangements of two Bisquare Blocks as well as of a Bisquare Block combined with a Versatile Block.

\begin{figure}[H]
    \centering
    \begin{subfigure}{0.32\textwidth}
        \includegraphics[scale=0.35]{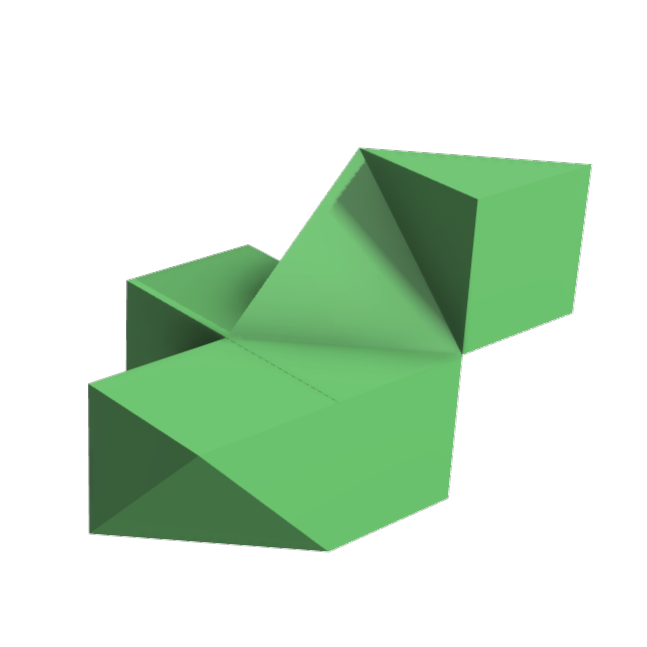}
        \caption{}
    \end{subfigure}
    \begin{subfigure}{0.32\textwidth}
        \includegraphics[scale=0.58]{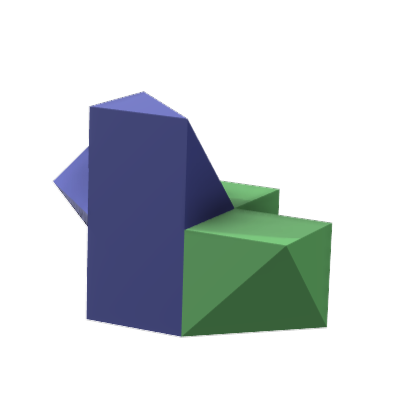}
        \caption{}
    \end{subfigure}
    \begin{subfigure}{0.32\textwidth}
        \includegraphics[scale=0.6]{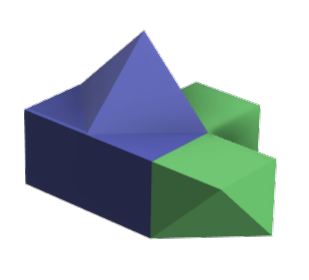}
        \caption{}
    \end{subfigure}
    \caption{Possible non-planar arrangements of two Bisquare Blocks (a) as well as of a Bisquare Block combined with a Versatile Block (b) and (c).}
    \label{fig:3D}
\end{figure}

By applying more general edge deformations to the square, such as smooth curves, additional intermediate points or asymmetric deformation paths, one obtains other blocks that can also be assembled together.
In \Cref{subfig:oct_sym_block,subfig:oct_asym_block} are two blocks depicted generated from a symmetric deformation of the square in which the deformation paths intersect the square’s edges.
These two blocks can be arranged in various ways. To determine which arrangements are admissible, we again translate them to square tiles subdivided into eight triangles. The triangles are coloured black and white, but adjacent triangles inside a square tile may now share the same colour, see \Cref{subfig:oct_sym,subfig:oct_asym}.

\begin{figure}[h!]
    \centering
     \begin{subfigure}{0.3\textwidth}
        \centering
        \includegraphics[scale=0.2]{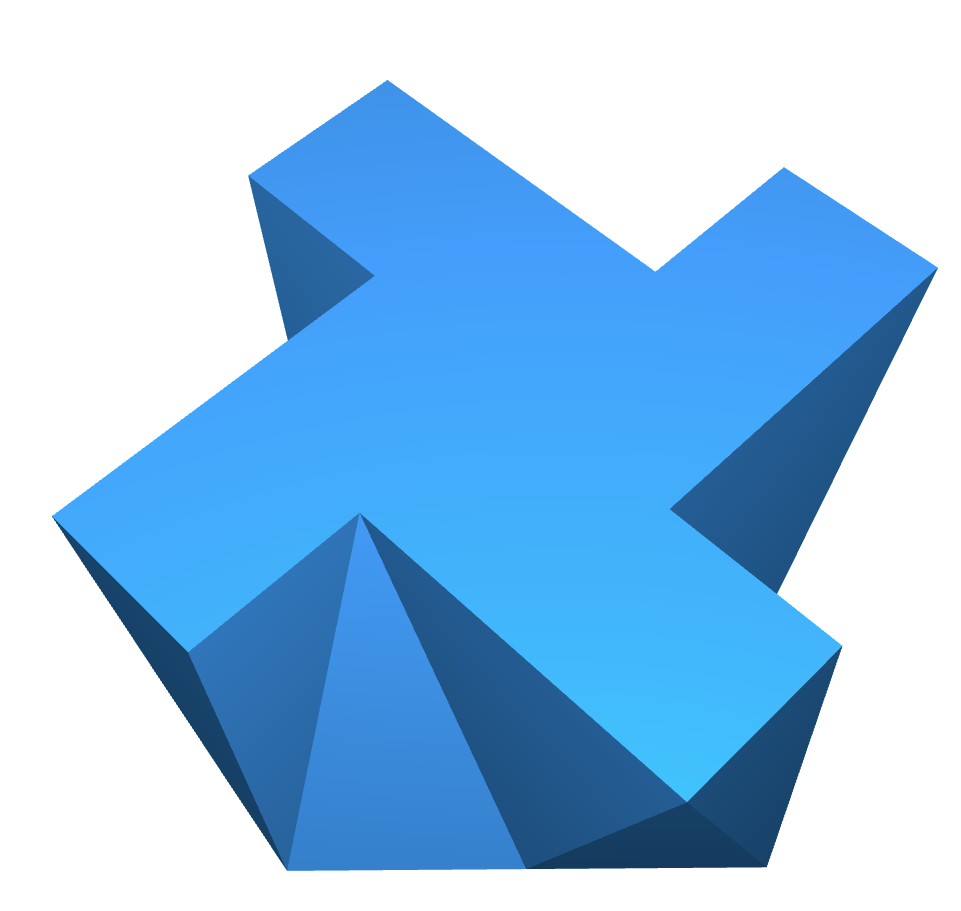}
        \caption{}
        \label{subfig:oct_sym_block}
    \end{subfigure}
    \begin{subfigure}{0.18\textwidth}
        \centering
        \raisebox{20pt}{
        \begin{tikzpicture}[scale=1.8]
            \fill[black] (0,0) -- (0.5,0) -- (0.5,0.5) -- cycle;
            \fill[black] (0,1) -- (0,0.5) -- (0.5,0.5) -- cycle;
            \fill[black] (1,1) -- (0.5,1) -- (0.5,0.5) -- cycle;
            \fill[black] (1,0) -- (1,0.5) -- (0.5,0.5) -- cycle;
        
            \draw[-,very thick] (0,0) to (1,0);
            \draw[-,very thick] (1,0) to (1,1);
            \draw[-,very thick] (1,1) to (0,1);
            \draw[-,very thick] (0,1) to (0,0);
        \end{tikzpicture}}
        \hspace{0.4cm}
        \caption{}
        \label{subfig:oct_sym}
    \end{subfigure}
    \begin{subfigure}{0.3\textwidth}
        \centering  
        \includegraphics[scale=0.22]{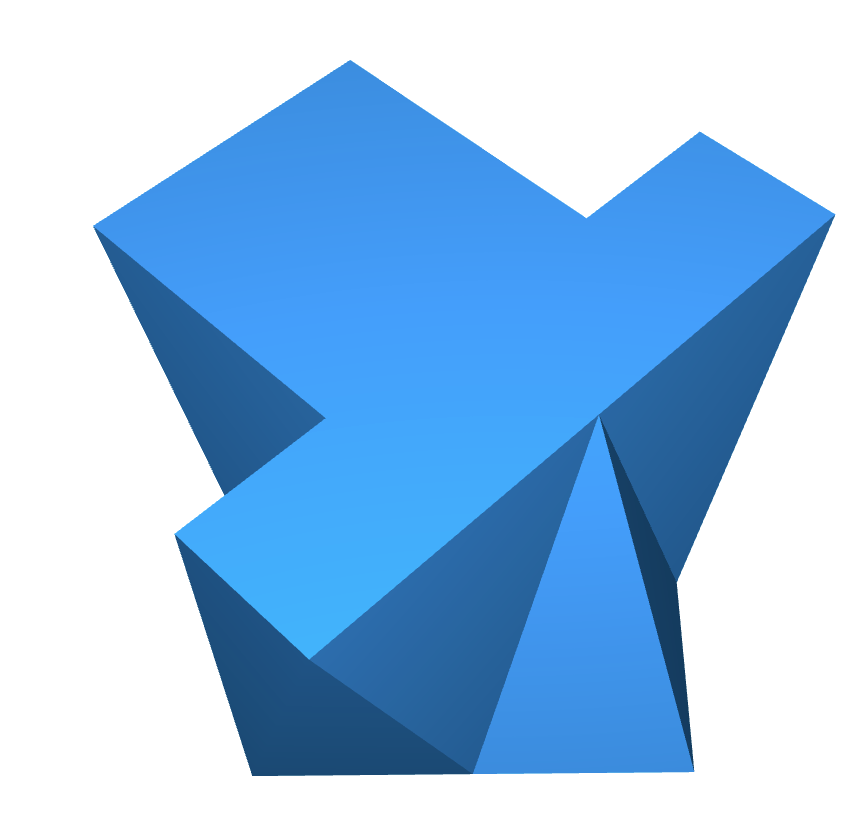}
        \caption{}
        \label{subfig:oct_asym_block}
    \end{subfigure}
    \begin{subfigure}{0.18\textwidth}
        \centering
        \raisebox{20pt}{
        \begin{tikzpicture}[scale=1.8]
            \fill[black] (0,0) -- (0.5,0.5) -- (0,0.5) -- cycle;
            \fill[black] (1,1) -- (0.5,0.5) -- (0.5,1) -- cycle;
            \fill[black] (0.5,0) -- (1,0) -- (1,0.5) -- (0.5,0.5) -- cycle;
        
            \draw[-,very thick] (0,0) to (1,0);
            \draw[-,very thick] (1,0) to (1,1);
            \draw[-,very thick] (1,1) to (0,1);
            \draw[-,very thick] (0,1) to (0,0);
        \end{tikzpicture}}
        \caption{}
        \label{subfig:oct_asym}
    \end{subfigure}
    \caption{Two blocks deformed by a zig-zag deformation path intersecting the edges of the fundamental domain (a) and (c) and their generalized Truchet tiles (b) and (d).}
    \label{fig:p4_inter}
\end{figure}
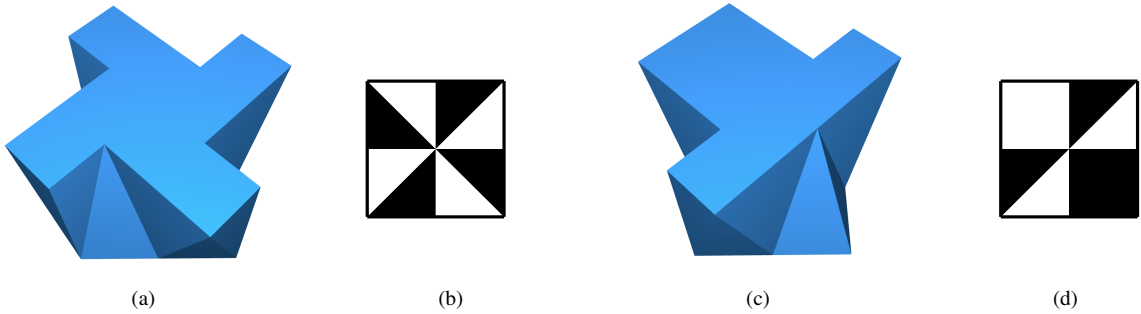

Tilings composed of the above tiles and satisfying the opposite‑adjacency condition correspond to valid arrangements of the blocks analogously to the bi- and quad-triangular tiles. A periodic example of such a tiling, together with the top view of its associated heterogenous assembly, is shown in \Cref{fig:zigzagAss}.

\begin{figure}[H]
    \centering
    \begin{subfigure}{0.4\textwidth}
        \centering
        \raisebox{10pt}{
        \input{ZigZagAssembly}}
        \caption{}
    \end{subfigure}
    \begin{subfigure}{0.4\textwidth}
        \centering
        \includegraphics[scale=0.35]{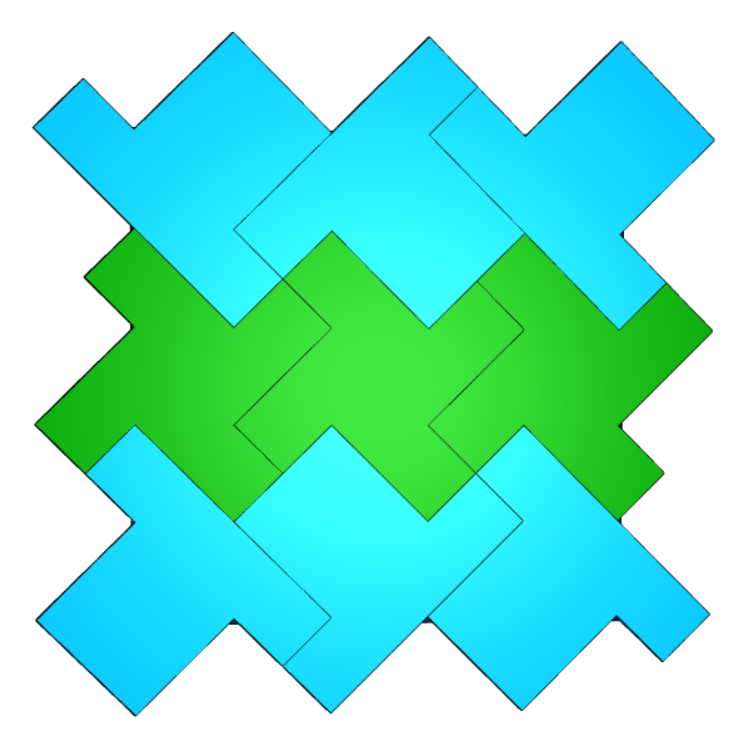}
        \caption{}
    \end{subfigure}
    \caption{A periodic tiling (a) and the corresponding heterogenous assembly of the two $p4$ blocks obtained by zig-zag deformation (b).}
    \label{fig:zigzagAss}
\end{figure}

Instead of piecewise linear curves it is also possible to use a quadratic deformation. Using the same deformation points as the Bisquare Block but with a quadratic rather than a linear deformation the block shown in \Cref{subfig:bisquare_curved} is obtained.
However, this block does not have an obverse block, as the quadratic deformation paths would intersect for this configuration.
By instead choosing a quadratic deformation whose intermediate points lie closer to the edges, a curved Abeille vault, designed by Truchet, is obtained and illustrated in \Cref{subfig:abeille}. For this block, an obverse block can be designed, as shown in \Cref{subfig:curvVersatile}.
\begin{figure}[h!]
    \centering
    \begin{subfigure}{0.35\textwidth}
        \centering
        \includegraphics[scale=0.23]{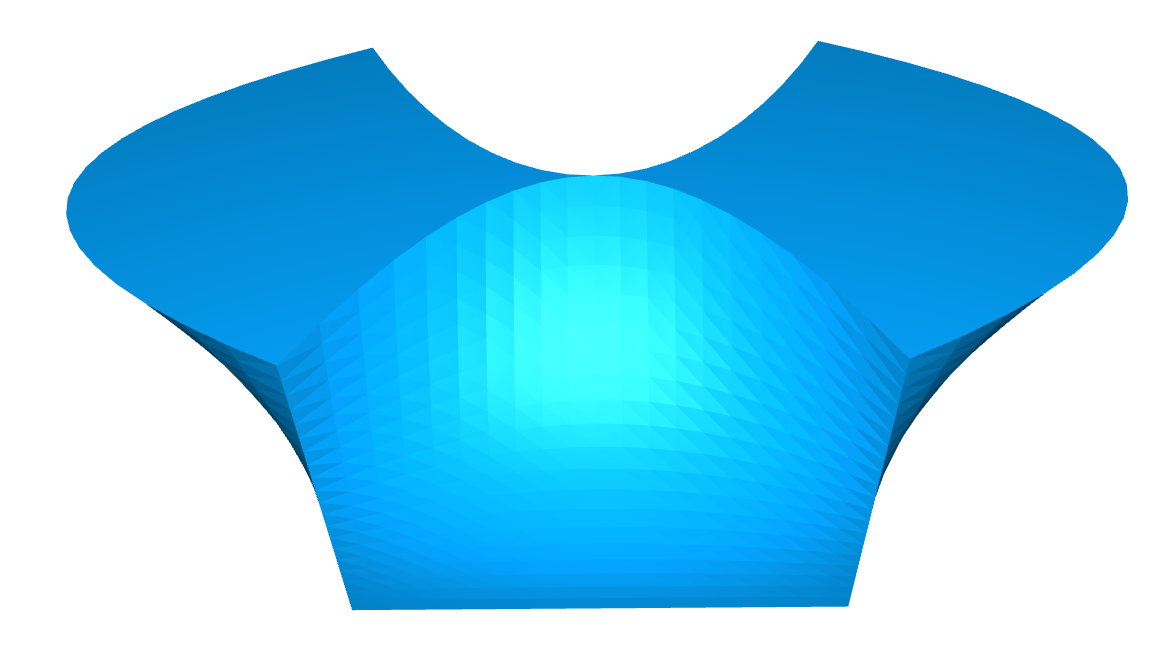}
        \caption{}
        \label{subfig:bisquare_curved}
    \end{subfigure}
    \begin{subfigure}{0.31\textwidth}
        \centering
        \includegraphics[scale=0.3]{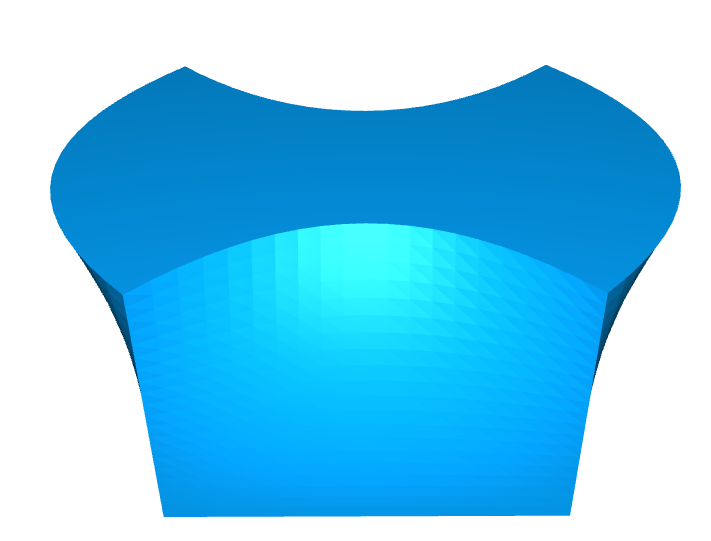}
        \caption{}
        \label{subfig:abeille}
    \end{subfigure}
    \begin{subfigure}{0.31\textwidth}
        \centering
        \includegraphics[scale=0.21]{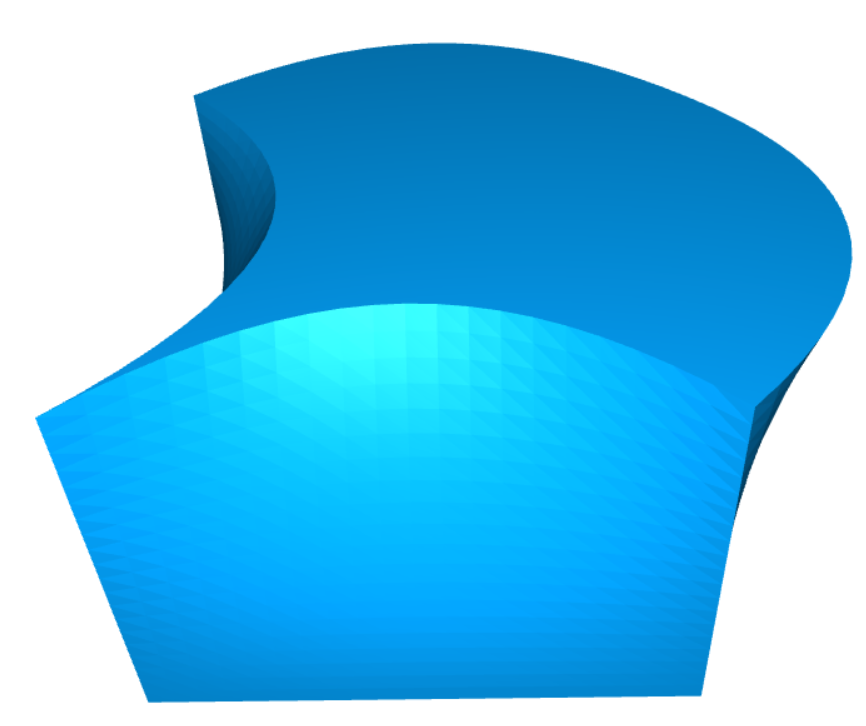}
        \caption{}
        \label{subfig:curvVersatile}
    \end{subfigure}
    \caption{Three blocks which can be constructed by the Escher trick and a quadratic deformation: A block analogous to the Bisquare Block (a), an Abeille vault (b) and its obverse block (c).}
    \label{fig:curved}
\end{figure}

The combinatorial structure derived from the tilings consisting of bi- and quad-triangular tiles remains valid even when the blocks are constructed by smooth deformations. 
In general, if two $p4$ blocks can be assembled together such that one corresponds to a bi‑triangular tile and the other to a quad‑triangular tile, the total number of admissible block arrangements remains unchanged.
Furthermore, an alternative fundamental domain to the square can be selected for the $p4$ group to create more heterogeneous space-filling assemblies, provided the domain has an even number of edges.

\subsection{Heterogenous assemblies based on \texorpdfstring{$p3$}{p3} symmetry}\label{sec:p3}

Another prominent class of tilings besides square tilings are lozenge tilings. They admit a natural correspondence with the $p3$ wallpaper group, since a lozenge can serve as a fundamental domain for $p3$.
A lozenge admits the same type of edge deformation as in the Versatile construction: we place a point at the midpoint of two adjacent edges not contained in the same edge pair and shift these points along the normal direction until the corresponding points of the opposite pair meet in the inner of the lozenge. The induced deformation and its obverse yield two distinct fundamental domains, as shown in \Cref{fig:lozengeDiff}.

\begin{figure}[H]
    \centering
    \input{Different_Deformations_Lozenge}
    \caption{Two compatible deformations of a lozenge both respecting $p3$ symmetry.}
    \label{fig:lozengeDiff}
\end{figure}
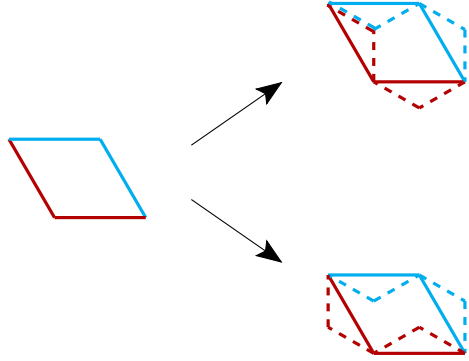

These two deformations of the lozenge give rise to two distinct blocks (see \Cref{fig:p3Blocks}). We refer to the left block of as the \emph{Rhom Block}, first introduced by~\cite{DissTom}, and to the right block as its obverse block.
They can be arranged either homogenously with $p3$ symmetry or together in a heterogenous assembly. 
 
\begin{figure}[h!]
    \centering
    \begin{subfigure}{0.4\textwidth}
        \centering
        \includegraphics[scale=0.52]{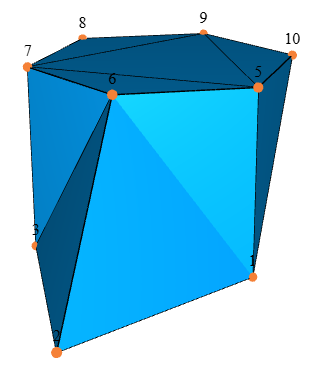}
        \vspace{-0.2cm}
        \caption{}
    \end{subfigure}
    \begin{subfigure}{0.4\textwidth}
        \centering
        \raisebox{10pt}{
        \includegraphics[scale=0.65]{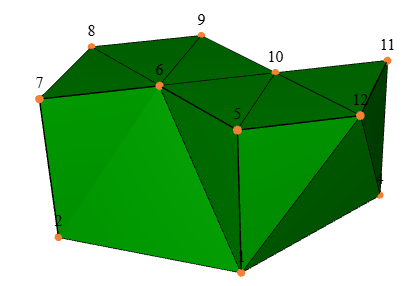}}
        \caption{}
    \end{subfigure}
    \caption{The Rhom Block (a) and its obverse block (b).}
    \label{fig:p3Blocks}
\end{figure}

The coordinates of the blocks are chosen so that the lozenge has edge length 1 and the blocks can also be assembled in a non-planar way.
With this the coordinates of the Rhom Block are defined as
$$\left(0,0,0\right),\left(\frac{1}{2},\frac{\sqrt{3}}{2},0\right),\left(1,0,0\right),\left(\frac{1}{2},-\frac{\sqrt{3}}{2},0\right),$$
$$\left(0,0,\frac{\sqrt{6}}{3}\right),\left(\frac{1}{2},\frac{\sqrt{3}}{6},\frac{\sqrt{6}}{3}\right),\left(1,0,\frac{\sqrt{6}}{3}\right),
\left(1,-\frac{\sqrt{3}}{3},\frac{\sqrt{6}}{3}\right),\left(\frac{1}{2},-\frac{\sqrt{3}}{2},\frac{\sqrt{6}}{3}\right),\left(0,-\frac{\sqrt{3}}{3},\frac{\sqrt{6}}{3}\right)$$
and the coordinates of its obverse block as
$$\left(0,0,0\right),\left(\frac{1}{2},\frac{\sqrt{3}}{2},0\right),\left(1,0,0\right),\left(\frac{1}{2},-\frac{\sqrt{3}}{2},0\right),\left(0,0,\frac{\sqrt{6}}{3}\right),\left(\frac{1}{2},\frac{\sqrt{3}}{6},\frac{\sqrt{6}}{3}\right)$$
$$\left(\frac{1}{2},\frac{\sqrt{3}}{2},\frac{\sqrt{6}}{3}\right),
\left(1,\frac{\sqrt{3}}{3},\frac{\sqrt{6}}{3}\right),\left(1,0,\frac{\sqrt{6}}{3}\right),\left(\frac{1}{2},-\frac{\sqrt{3}}{6},\frac{\sqrt{6}}{3}\right),\left(\frac{1}{2},-\frac{\sqrt{3}}{2},\frac{\sqrt{6}}{3}\right),\left(0,-\frac{\sqrt{3}}{3},\frac{\sqrt{6}}{3}\right).$$

Similar to the correspondence between square tilings and assemblies composed of Versatile and Bisquare Blocks, the Rhomb Block and its obverse block correspond to lozenge tilings.
Analogously, the lozenges are subdivided into two triangles for the Rhom Block and into four triangles for its obverse block as shown in \Cref{fig:lozenge}.

\begin{figure}[H]
    \centering
    \begin{subfigure}{0.4\textwidth}
        \raisebox{14pt}{
        \input{Lozenge_Rhom}}
        \includegraphics[scale=0.25]{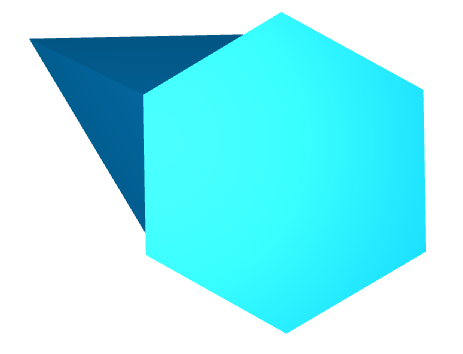}
        \caption{}
    \end{subfigure}
    \begin{subfigure}{0.4\textwidth}
        \raisebox{-80pt}{
        \input{Lozenge_p3}}
        \raisebox{-100pt}{\includegraphics[scale=0.3]{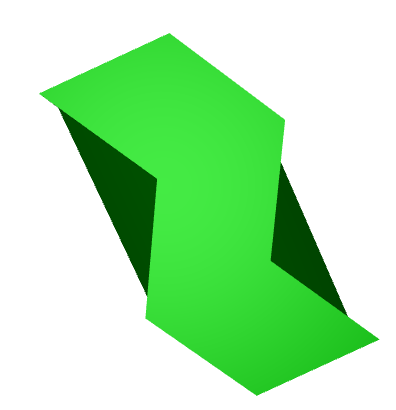}}
        \caption{}
    \end{subfigure}
    \caption{The correspondence between a lozenge and the Rhom Block (a) and its obverse block (b).}
    \label{fig:lozenge}
\end{figure}

In the case of $p4$ symmetry with a square as the fundamental domain, there is a unique square tiling that must be suitably decorated with the bi- and quad-triangular tiles to obtain an assembly of the Versatile and Bisquare Block.
However, there is a wide variety of lozenge tilings, as each lozenge can be rotated by $120^\circ$. This means that lozenges have three different orientations whereas the square has only one. As shown by~\cite{Gorin2021}, the number of possible lozenge tilings of a $a\times b\times c$ hexagon is $$\prod_{i=1}^a\prod_{j=1}^b\prod_{k=1}^c\frac{a+b+c-1}{a+b+c-2}.$$
For all of these lozenge tilings, many different decorations of the tiles depicted in \Cref{fig:lozenge}, are possible. Every decoration of a lozenge tiling that satisfies the opposite‑colour adjacency condition can be translated into an assembly of the Rhom Block and its obverse block.
In \Cref{subfig:aperiodic_p3} a possible aperiodic decoration of the $p3$ lozenge tiling forming a $3\times 3\times 3$ hexagon is shown, with the corresponding assembly depicted in \Cref{subfig:aperiodic_ass}.

\begin{figure}[H]
    \centering
    \begin{subfigure}{0.45\textwidth}
        \centering
        \raisebox{15pt}{
        \input{LozengeTiling1}}
        \caption{}
        \label{subfig:aperiodic_p3}
    \end{subfigure}
    \begin{subfigure}{0.45\textwidth}
        \centering
        \includegraphics[scale=0.3]{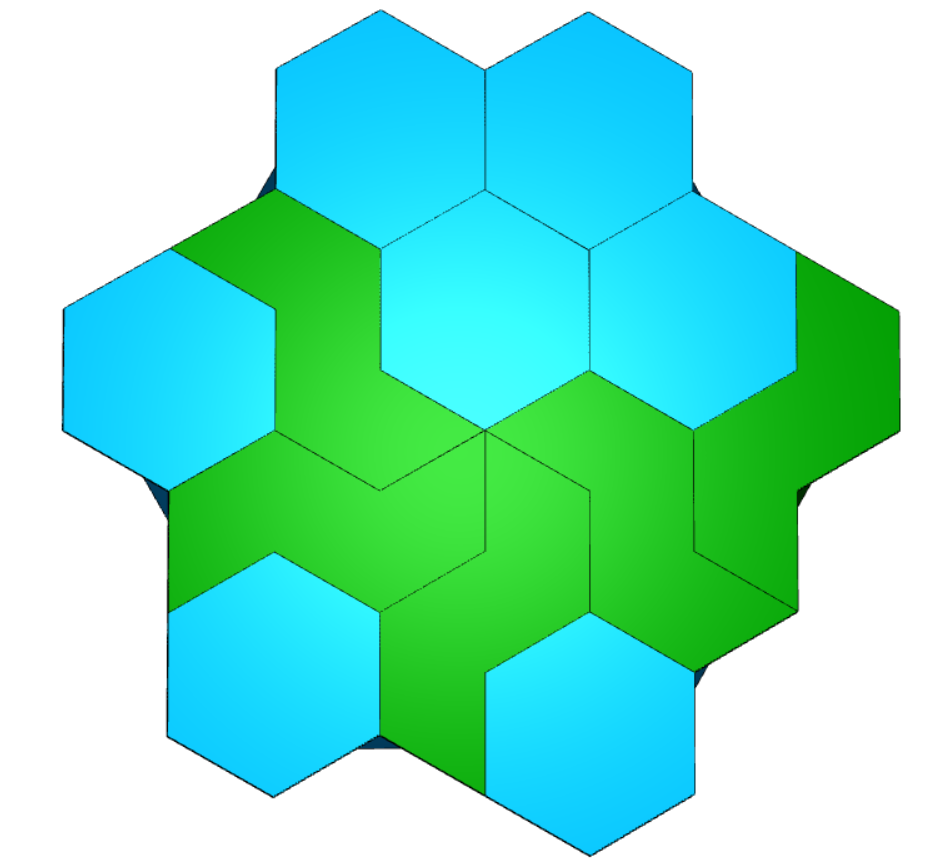}
        \caption{}
        \label{subfig:aperiodic_ass}
    \end{subfigure}
    \caption{An aperiodic decorated $p3$ lozenge tiling of a $3\times 3\times 3$ hexagon (a) and the corresponding heterogenous assembly containing the Rhom Block and its obverse block (b).}
    \label{fig:lozengeTilings}
\end{figure}

As in the case of the $p4$ group, both zig‑zag and smooth deformations can be used to generate further examples of interlocking blocks, as well as alternative fundamental domains with an even number of edges.

\subsection{Heterogenous assemblies based on \texorpdfstring{$p6$}{p6} symmetry}
For the wallpaper group $p6$, whose symmetry includes $60^{\circ}$ rotations, a lozenge can likewise be taken as a fundamental domain, analogous to the $p3$ case, but with a different centre of rotation. Hence, the bottom deformation introduced in \Cref{sec:p3} remain valid for the $p6$ group.
Another possible choice for the fundamental domain of the wallpaper group $p6$ is a kite.
Its edges admit the same type of deformation as in the Versatile construction: we place a point at the midpoint of two adjacent edges not contained in the same edge pair and shift these points along the normal direction until the corresponding points of the opposite pair meet. The induced deformation and its obverse yield two distinct fundamental domains, as shown in \Cref{fig:kite}.

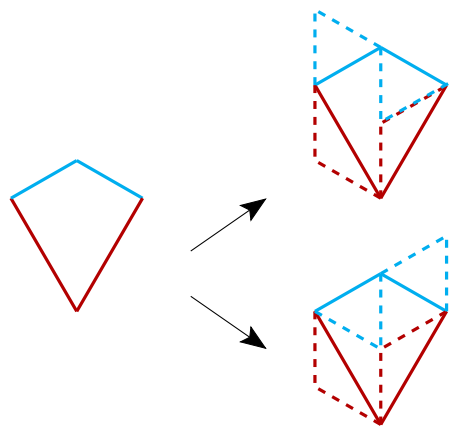
\begin{figure}[h!]
    \centering
    \input{Different_Deformations_Kite}
    \caption{Two compatible deformations of a kite both respecting $p6$ symmetry.}
    \label{fig:kite}
\end{figure}

These deformations give rise to two distinct blocks: a sheared variant of the Versatile Block (see \Cref{subfig:versatilep6}) and a block whose top face consists of two lozenges meeting at a vertex (see \Cref{subfig:bilozenge}), similar to the two squares of the Bisquare Block. Both blocks can be assembled homogenously using the wallpaper group $p6$, or they may be combined heterogenously within an assembly. In analogy to the decorations of the square and lozenge tilings, the blocks can again be related to tiles, here kites, and equipped with a bi‑colouring that reflects their respective orientations, see \Cref{subfig:versatilep6_tile,subfig:bilozenge_tile}.

\begin{figure}[H]
    \centering
    \begin{subfigure}{0.3\textwidth}
        \centering
        \includegraphics[scale=0.35]{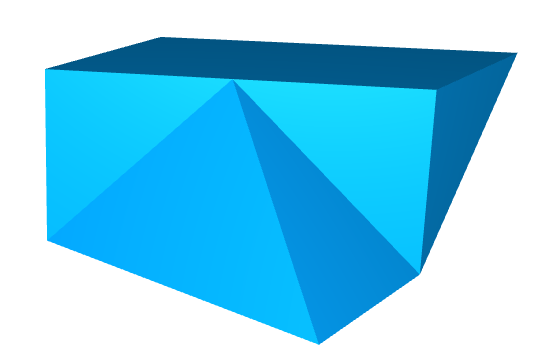}
        \caption{}
        \label{subfig:versatilep6}
    \end{subfigure}
    \begin{subfigure}{0.18\textwidth}
        \centering
        \raisebox{20pt}{\input{versatile_p6_tile}}
        \caption{}
        \label{subfig:versatilep6_tile} 
    \end{subfigure}
    \begin{subfigure}{0.3\textwidth}
        \centering
        \raisebox{20pt}{
        \includegraphics[scale=0.4]{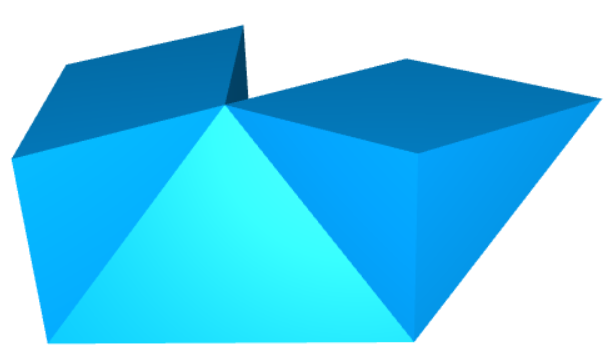}}
        \caption{}
        \label{subfig:bilozenge}
    \end{subfigure}
    \begin{subfigure}{0.18\textwidth}
        \centering
        \raisebox{20pt}{
        \input{Bilozenge_tile}}
        \caption{}
        \label{subfig:bilozenge_tile} 
    \end{subfigure}
    \caption{The $p6$ variants of the Versatile Block (a) and the Bisquare Block (c) and their corresponding tiles (b) + (d).}
    \label{fig:p6Blocks}
\end{figure}

As with the other feasible wallpaper groups, both zig‑zag and smooth deformations can be used to generate further examples of blocks admitting TIAs.

\subsection{Example for \texorpdfstring{$p1$}{p1} and \texorpdfstring{$p2$}{p2} symmetry}
In the case of the wallpaper groups $p1$ and $p2$, we begin with a square as the fundamental domain. For both groups, opposite edges form an edge pair, but in $p1$ they are mapped to each other by a translation, whereas in $p2$ the mapping is given by a $180^\circ$ rotation. 
A deformation preserving both groups is a symmetric zig‑zag path resulting in two different deformations of the square, as shown in \Cref{subfig:p1_p2_deform}.
This yields the block depicted in \Cref{subfig:oct_sym_block}, which can be arranged with $p1$, $p2$, and $p4$ symmetry and the obverse block, shown in \Cref{subfig:p1_p2_block}, that can be arranged according to the $p1$ and $p2$ groups. As in the previous cases, both blocks can also be combined to form a heterogeneous assembly, which, analogously to the $p4$ setting, relates to Truchet tilings.

\begin{figure}[h!]
    \centering
    \begin{subfigure}{0.55\textwidth}
        \centering
        \hspace{-2cm}
        \input{Different_Deformations_p1}
        \caption{}
        \label{subfig:p1_p2_deform}
    \end{subfigure}
    \begin{subfigure}{0.4\textwidth}
        \centering
        \includegraphics[scale=0.35]{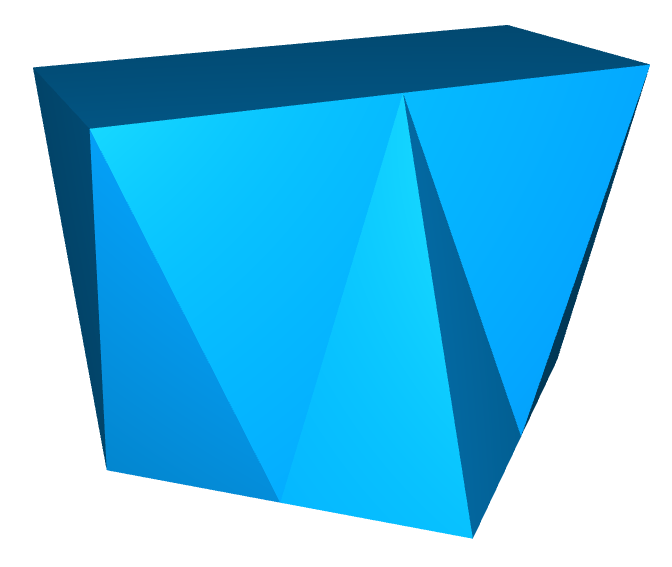}
        \caption{}
        \label{subfig:p1_p2_block}
    \end{subfigure}
    \caption{Symmetric zig-zag deformation preserving $p1$ and $p2$ (a) and the block resulting from the upper deformation (b).}
\end{figure}

\subsection{Heterogenous assemblies with glide reflection}

The group $pg$ consists solely of translations and glide reflections, each of which is a reflection combined with a translation. A convenient choice of a fundamental domain is a square equipped with a vertical glide reflection.
In this setup, the vertical edges of the square are identified by a translation, while the horizontal edges are identified by a reflection followed by a translation.
We again deform the edges using symmetric zig‑zag paths, as illustrated in \Cref{subfig:pg_deform}. This construction yields a fundamental domain $F$ whose vertical reflection does not yield a rotated copy of $F$. Consequently, arranging the block $B$ derived from $F$ (see \Cref{subfig:pg_block1}) in $pg$ symmetry requires a second block, resulting from the vertical reflection of $B$ (see \Cref{subfig:pg_block2}). This demonstrates that, for certain deformations in the $pg$ group case, a heterogeneous assembly becomes directly necessary. However, the second block required for the assembly is precisely the one obtained by reversing the order in which the vertical edges of the square are deformed—analogous to the construction used for the Versatile Block and the Bisquare Block.

\begin{figure}[h!]
    \centering
    \begin{subfigure}{0.35\textwidth}
        \centering
        \hspace{-4cm}
        \input{pg_deformation}
        \caption{}
        \label{subfig:pg_deform}
    \end{subfigure}
    \begin{subfigure}{0.26\textwidth}
        \centering
        \includegraphics[scale=0.3]{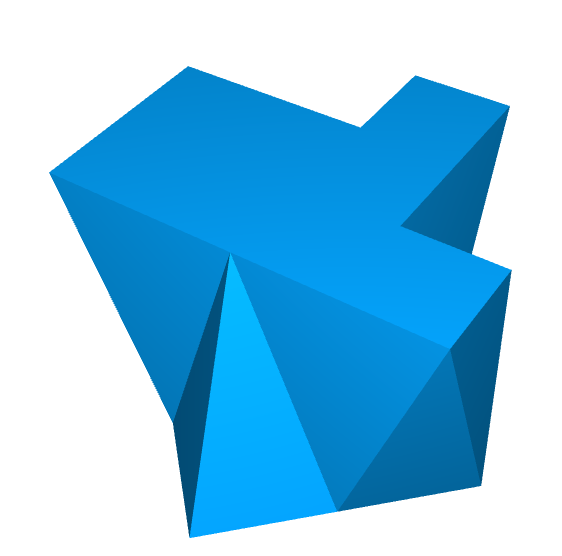}
        \caption{}
        \label{subfig:pg_block1}
    \end{subfigure}
    \begin{subfigure}{0.26\textwidth}
        \centering
        \includegraphics[scale=0.25]{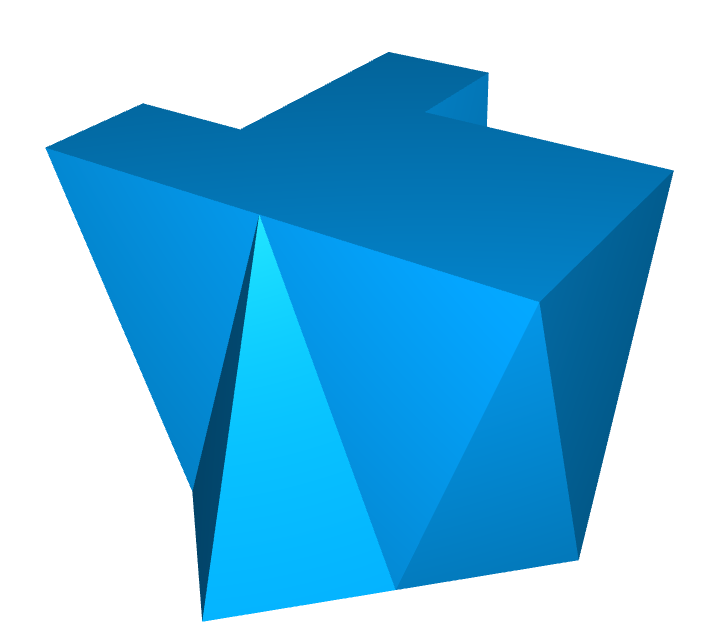}
        \caption{}
        \label{subfig:pg_block2}
    \end{subfigure}
    \caption{Symmetric zig-zag deformation preserving $pg$ symmetry (a), the resulting block (b) and the reflected/obverse block (c).}
\end{figure}

The last wallpaper group of the seven groups that can be used to construct topological interlocking assemblies is $p2gg$. This group contains two glide reflections in perpendicular directions and two rotations of $180^\circ$.
A square can be chosen again as the fundamental domain where the vertical and horizontal edges are identified by a glide reflection. 
To achieve a truly heterogeneous assembly, the different edge pairs must be deformed by different deformation types.
One edge pair is deformed symmetrically by introducing an intermediate point at the midpoint of the edge, whereas the other edge pair is deformed by an intermediate point corresponding to an end vertex that is shifted in the normal direction of the edge. The resulting deformation is illustrated in \Cref{subfig:p2gg_deform}.
Analogously to the example of the $pg$ case we directly obtain two blocks that are necessary for a $p2gg$ arrangement, see \Cref{subfig:p2gg_block1,subfig:p2gg_block2}.

\begin{figure}[H]
    \centering
    \begin{subfigure}{0.35\textwidth}
        \centering
        \hspace{-4cm}
        \input{p2gg_deformation}
        \caption{}
        \label{subfig:p2gg_deform}
    \end{subfigure}
    \begin{subfigure}{0.26\textwidth}
        \centering
        \includegraphics[scale=0.25]{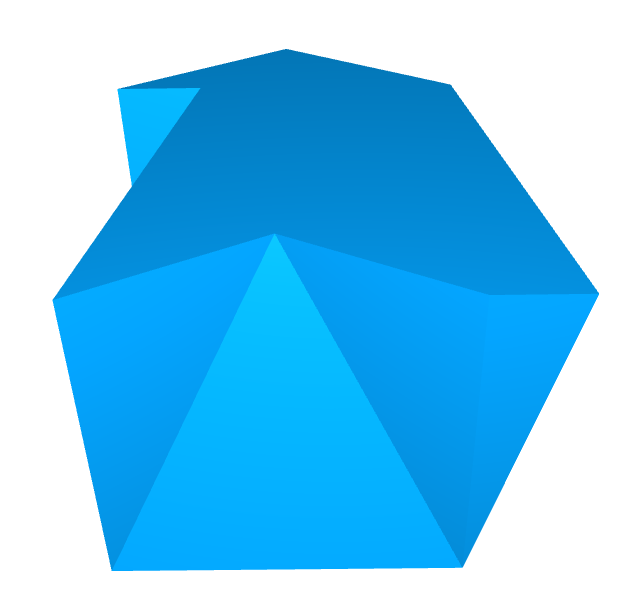}
        \caption{}
        \label{subfig:p2gg_block1}
    \end{subfigure}
    \begin{subfigure}{0.26\textwidth}
        \centering
        \includegraphics[scale=0.2]{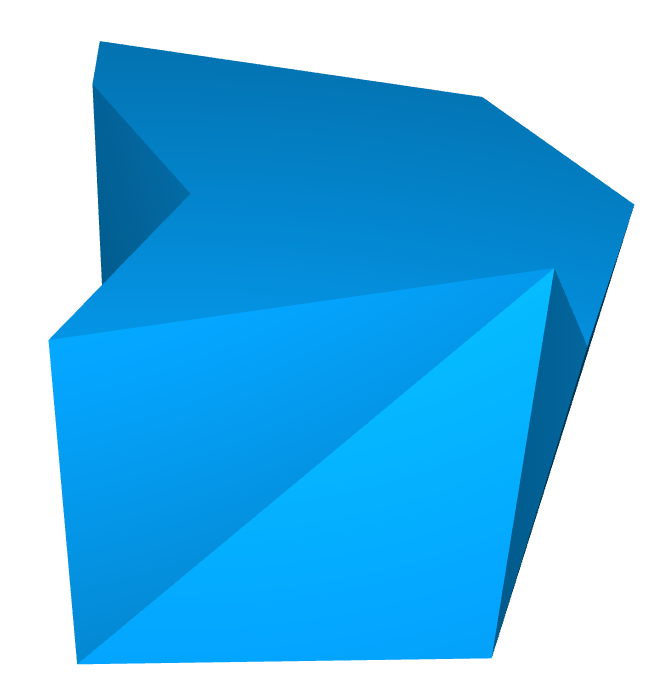}
        \caption{}
        \label{subfig:p2gg_block2}
    \end{subfigure}
    \caption{Asymmetric deformation preserving $p2gg$ symmetry (a), the resulting block (b) and the reflected block (c).}
\end{figure}

\section{Using Semiregular Tessellations}

Heterogenous and space-filling assemblies can also be constructed from tilings composed of multiple polygons, rather than restricting oneself to fundamental domains of wallpaper groups.
In such cases, for example for semiregular tessellations, the Escher Trick must be applied to several polygons simultaneously so that their deformed versions remain compatible.
A \emph{semiregular tessellation} is a tessellation of the plane by at least two regular polygons arranged so that the same polygons appear in the same cyclic order around every vertex. There exist exactly eight semiregular tessellations of the plane~\citep{TilingsAndPatterns}.

The \emph{snub square tiling} is a semiregular tessellation consisting of triangles and squares, in which the cyclic ordering of polygons around each vertex is two adjacent triangles, followed by a square, then two triangles, and finally another square. 
Since two triangles always occur as a pair, they can be combined to a lozenge, and the resulting tiling is shown in \Cref{subfig:snub_square}.
The snub square tiling exhibits a $p4g$ symmetry, characterized by $90^\circ$ rotations and glide reflections. Since each edge is shared by a square and a lozenge, any chosen deformation curve necessarily constrains the deformation of both, the square and the lozenge.

As an example, the lozenge can be deformed to a hexagon as for the Rhom Block. In this case, the squares must be deformed in a manner similar to the Versatile Block, but with the intermediate point placed closer to the edge.
These deformations are depicted in \Cref{subfig:deform_rhom} and preserve the $90^{\circ }$ rotational symmetry and the glide‑reflection symmetry due to their inherent symmetry.
Another possibility is to deform the squares in the same manner as for the Versatile Block. In this case, the lozenges must be deformed as shown in \Cref{subfig:deform_versatile}, since otherwise the deformation paths would intersect, which is not permitted.

\begin{figure}[h!]
    \centering
    \begin{subfigure}{0.32\textwidth}
        \centering 
        \includegraphics[scale=0.3]{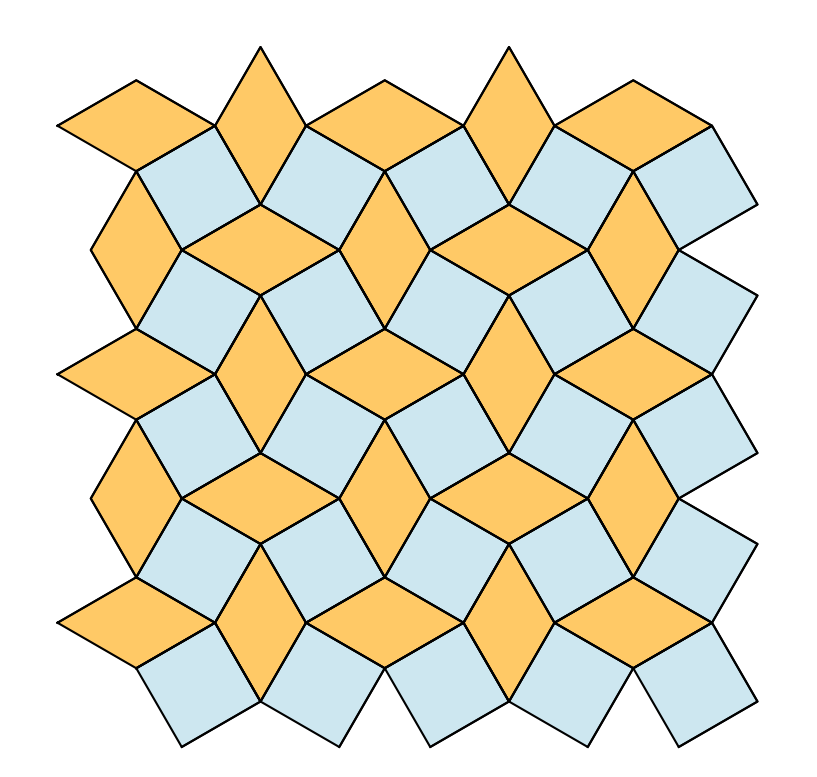}
        \caption{}
        \label{subfig:snub_square}
    \end{subfigure}
    \begin{subfigure}{0.32\textwidth}
        \centering
        \input{DeformationRhom}
        \caption{}
        \label{subfig:deform_rhom}
    \end{subfigure}
    \begin{subfigure}{0.32\textwidth}
        \centering
        \input{DeformationVersatile}
        \caption{}
        \label{subfig:deform_versatile}
    \end{subfigure}
    \caption{The snub square tiling with two triangles joined to a lozenge (a) and two deformations of a lozenge and a neighbouring square: Based on the deformation of the Rhom Block (b) and the Versatile Block (c).}
    \label{fig:deform_snub}
\end{figure}

The assembly based on the snub square tiling with the Rhom Block and its obverse block is depicted in \Cref{subfig:rhom_top} where the Rhom Block is coloured blue and the obverse block red.
In the other case, where the square is deformed as for the Versatile Block, the obverse block for the lozenge no longer preserves the glide‑reflection symmetry.
Thus, to obtain a suitable arrangement based on the snub square tiling, we also need to include the reflection of the block resulting from the deformation of the lozenge. This yields a heterogenous arrangement of three different blocks, see \Cref{subfig:versatile_top}. Here, the Versatile Block is coloured red and the two blocks corresponding to the lozenges are coloured blue and green.

\begin{figure}[h!]
    \centering
    \begin{subfigure}{0.45\textwidth}
        \centering
        \includegraphics[scale=0.35]{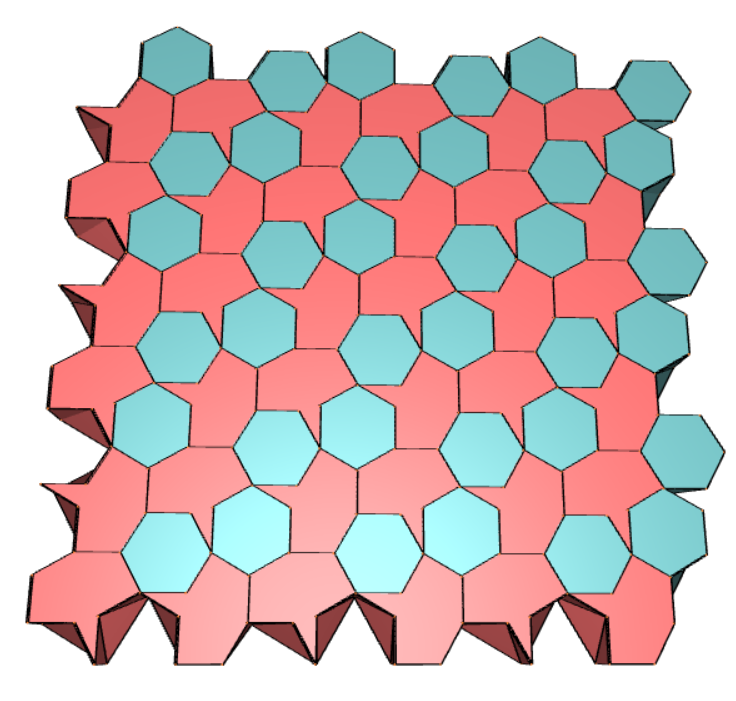}
        \caption{}
        \label{subfig:rhom_top}
    \end{subfigure}
    \begin{subfigure}{0.45\textwidth}
        \centering
        \includegraphics[scale=0.4]{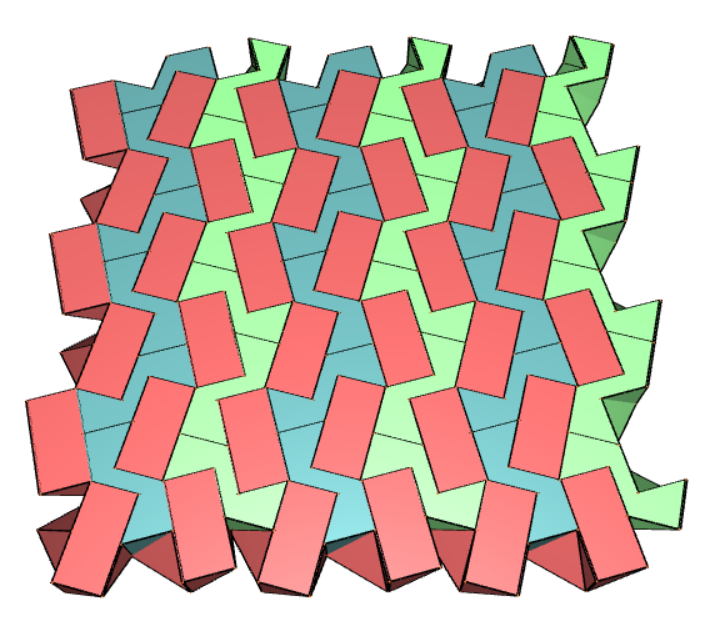}
        \caption{}
        \label{subfig:versatile_top}
    \end{subfigure}
    \caption{The assembly based on the snub square tiling with the Rhom Block (a) and the assembly based on the snub square tiling with the Versatile Block (b).}
    \label{fig:snub_square}
\end{figure}

Similar to the bi‑ and quad‑triangular squares and the decorated lozenges, the squares and lozenges of the snub square tiling can also be coloured to determine which arrangements are possible when satisfying the opposite‑colour adjacency condition.
Analogously, to the construction used for the snub square tiling, all other semiregular tilings can be employed in the same way.
Moreover, smooth deformations can likewise be applied in the case of semiregular tilings.

\section{Conclusion}

Many different examples of blocks that form an assembly consisting of one or more block types have been presented. By choosing the blocks at the perimeter as the frame these assemblies are heterogeneous space-filling topological interlocking assemblies.
This illustrates the breadth of possibilities offered by the method based on the Escher Trick. In most cases, piecewise linear paths were used, but—as shown in \Cref{fig:curved}—non‑linear deformations are equally feasible. For most wallpaper groups we applied the same edge deformation to all edge pairs, but the design space expands significantly when different deformations are assigned to individual edge pairs. However, this reduces the symmetry of the blocks, which may be a disadvantage depending on the intended application.
Moreover, the size of the blocks can be scaled for the intended application, and the distance between the two fundamental domains, i.e.\ the height of the block, can be chosen independently. For improved material efficiency, all these blocks can be made hollow by removing the top and bottom faces.

To obtain an even larger variety of blocks, including ones with potentially improved mechanical performance due to stronger interlocking, we can apply the Escher Trick twice~\citep{DissTom}. For example, we may start with a square, deform it into two distinct fundamental domains (one placed above and one placed below the square), and then interpolate between them, see \Cref{fig:doubleEscher} for examples.
\Cref{subfig:abeille_double} shows an Abeille vault which can be constructed by applying the Escher Trick twice. Applying the deformation used for the Versatile Block twice to the square produces the block depicted in \Cref{subfig:doubleVersatile}. In \Cref{subfig:versatile_bisquare}, the square is deformed once using the Versatile deformation and once using the Bisquare deformation. The resulting block admits both a homogeneous assembly and a heterogeneous assembly when combined with the block constructed by two Versatile deformations.

\begin{figure}[h!]
    \centering
    \hspace{-2cm}
    \begin{subfigure}{0.32\textwidth}
        \centering
        \raisebox{10pt}{
        \includegraphics[scale=0.35]{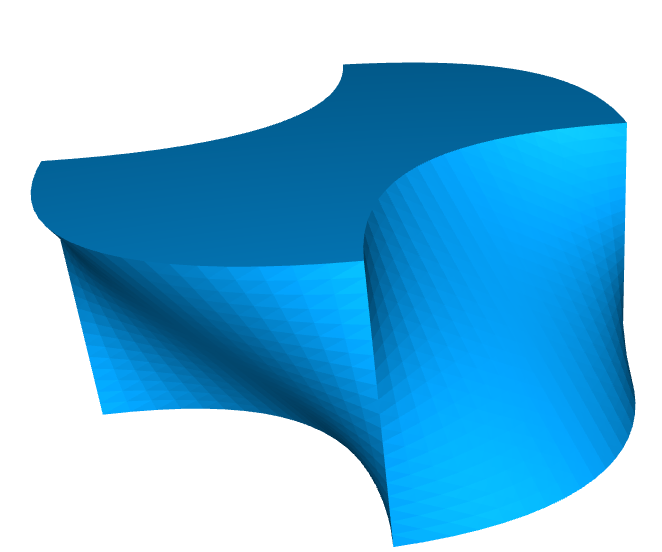}}
        \caption{}
        \label{subfig:abeille_double}
    \end{subfigure}
    \begin{subfigure}{0.32\textwidth}
        \centering
        \includegraphics[scale=0.4]{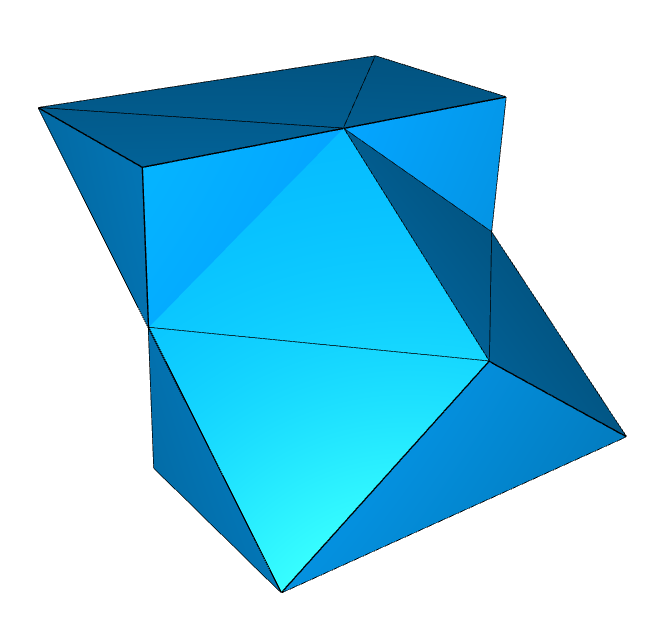}
        \caption{}
        \label{subfig:doubleVersatile}
    \end{subfigure}
    \begin{subfigure}{0.32\textwidth}
        \centering
        \raisebox{10pt}{
        \includegraphics[scale=0.6]{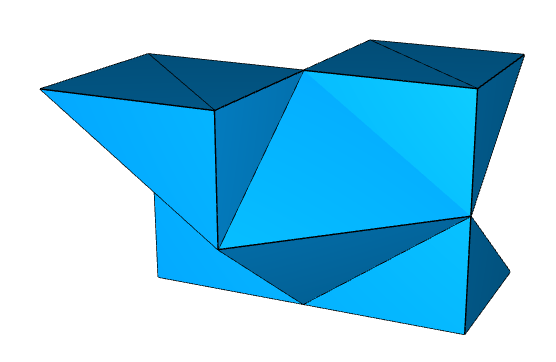}}
        \caption{}
        \label{subfig:versatile_bisquare}
    \end{subfigure}
    \caption{Blocks obtained from applying the Escher Trick twice based on the Abeille vault (a), the Versatile Block (b) and the Versatile Block and the Bisquare Block (c).}
    \label{fig:doubleEscher}
\end{figure}

All these possibilities—both in the design of individual blocks and in their arrangements—show that the choice of a topological interlocking configuration must be made with care.
Due to our experience we expect that the different blocks exhibit distinct mechanical behaviours, and the symmetries or asymmetries they possess play a crucial role in determining their mechanical performance. A detailed assessment of these effects will require dedicated multibody‑dynamics simulations and finite‑element analyses \citep{InfluenceArrangement}.

\section*{Acknowledgements}
We acknowledge the funding by the Deutsche Forschungsgemeinschaft (DFG, German Research Foundation) in the framework of the Collaborative Research Centre CRC/TRR 280 “Design Strategies for Material-Minimized Carbon Reinforced Concrete Structures – Principles of a New Approach to Construction” (project ID 417002380).

\bibliography{main.bib}

\end{document}

%% file: Example_pattern.tex
\begin{tikzpicture}[scale=1.]
\tikzset{knoten/.style={circle,fill=black,inner sep=0.2mm}}
\foreach \tx in {1,...,2}
    \foreach \ty in {1,...,2}
        {\begin{scope}[shift={(2*\tx,2*\ty)}]
            \foreach \x in {1,...,4}
                {\begin{scope}[rotate=90*\x]
                    \draw[-, thick] (0,0) to (0,1);
                    \draw[-, thick] (0,0) to (1,0);
                    \draw[-, thick] (1,1) to (0,1);
                    \draw[-, thick] (1,1) to (1,0);

                    \draw[-, thick] (0,0.5) to (0.5,0);

                    \draw (0.5,0.5) circle (0.2cm);
        \end{scope}}
    \end{scope}}

    \draw[-,red!60!black,very thick] (1,1) to (1,3);
    \draw[-,red!60!black,ultra thick] (1,1) to (3,1);
    \draw[-,red!60!black,ultra thick] (3,3) to (1,3);
    \draw[-,red!60!black,ultra thick] (3,3) to (3,1);

    \draw[dashed,cyan,ultra thick] (1,1) to (1,2);
    \draw[dashed,cyan,ultra thick] (1,1) to (2,1);
    \draw[dashed,cyan,ultra thick] (2,2) to (1,2);
    \draw[dashed,cyan,ultra thick] (2,2) to (2,1);
\end{tikzpicture}

%% file: Construction_Versatile.tex
\begin{minipage}{0.25\textwidth}
\begin{tikzpicture}[scale=1.5]
    \tikzset{knoten/.style={circle,fill=black,inner sep=0.5mm}}

    \draw[-,red!70!black,very thick] (0,0) to (0,1);
    \draw[-,red!70!black,very thick] (0,0) to (1,0);
    \draw[-,cyan,very thick] (1,1) to (0,1);
    \draw[-,cyan,very thick] (1,1) to (1,0);
    \draw [arrows = {-Stealth[length=10pt, inset=2pt]}] (1.5,0.5) -- (2.5,0.5);

    \node[knoten] (x) at (0,0) {};
    \node[knoten] (y) at (0,1) {};
    \node[knoten] (z) at (1,1) {};
    \node[knoten] (w) at (1,0) {};

    \node (a) at (1,-0.3) {$v_1$};
    \node (b) at (0,-0.3) {$v_2$};
    \node (c) at (1,1.25) {$v_4$};
    \node (d) at (0,1.25) {$v_3$};
\end{tikzpicture}
\end{minipage}
\begin{minipage}{0.3\textwidth}
\begin{tikzpicture}[scale=1.5]
    \tikzset{knoten/.style={circle,fill=black,inner sep=0.5mm}}

     \draw[-,cyan,very thick,dashed] (0.52,0.5) to (1,0.02);
     \draw[-,red!70!black,very thick,dashed] (0.48,0.5) to (0.98,0);

    \draw[-,red!70!black,very thick] (0,0) to (1,0);
    \draw[-,red!70!black,very thick] (0,0) to (0,1);
    \draw[-,cyan,very thick] (1,1) to (1,0);
    \draw[-,cyan,very thick] (1,1) to (0,1);

    \draw[-,cyan,very thick,dashed] (0.5,1.5) to (1,1);
    \draw[-,cyan,very thick,dashed] (0.5,1.5) to (0,1);
    \draw[-,cyan,very thick,dashed] (0.5,0.5) to (1,1);

    \draw[-,red!70!black,very thick,dashed] (0.5,0.5) to (0,0);
    \draw[-,red!70!black,very thick,dashed] (-0.5,0.5) to (0,0);
    \draw[-,red!70!black,very thick,dashed] (-0.5,0.5) to (0,1);
    \draw [arrows = {-Stealth[length=10pt, inset=2pt]}] (1.5,0.75) -- (2.5,0.75);

    \node[knoten] (x) at (0,0) {};
    \node[knoten] (y) at (0,1) {};
    \node[knoten] (z) at (1,1) {};
    \node[knoten] (w) at (1,0) {};

    \node[knoten] (k) at (-0.5,0.5) {};
    \node[knoten] (l) at (0.5,0.5) {};
    \node[knoten] (l) at (0.5,1.5) {};

    \node (a) at (1,-0.3) {$v_1$};
    \node (b) at (0,-0.3) {$v_2$};
    \node (c) at (1,1.25) {$v_4$};
    \node (d) at (0,1.25) {$v_3$};
    \node (e) at (0.3,0.5) {$p$};
\end{tikzpicture}
\end{minipage}
\begin{minipage}{0.3\textwidth}
\begin{tikzpicture}[scale=1.5]
    \foreach \x in {1,...,4}
    {\begin{scope}[rotate=90*\x]
         \draw[-,cyan,very thick,dashed] (0.52,0.5) to (1,0.02);
         \draw[-,red!70!black,very thick,dashed] (0.48,0.5) to (0.98,0);

        \draw[-,red!70!black,very thick] (0,0) to (1,0);
        \draw[-,red!70!black,very thick] (0,0) to (0,1);
        \draw[-,cyan,very thick] (1,1) to (1,0);
        \draw[-,cyan,very thick] (1,1) to (0,1);
    
        \draw[-,cyan,very thick,dashed] (0.5,1.5) to (1,1);
        \draw[-,cyan,very thick,dashed] (0.5,1.5) to (0,1);
        \draw[-,cyan,very thick,dashed] (0.5,0.5) to (1,1);
    
        \draw[-,red!70!black,very thick,dashed] (0.5,0.5) to (0,0);
        \draw[-,red!70!black,very thick,dashed] (-0.5,0.5) to (0,0);
        \draw[-,red!70!black,very thick,dashed] (-0.5,0.5) to (0,1);
    \end{scope}}
\end{tikzpicture}
\end{minipage}

%% file: bottom_left_tile.tex
\begin{tikzpicture}[scale=1.2]
    \fill[black] (0,0) -- (1,0) -- (0,1) -- cycle;

    \draw[-,very thick] (0,0) to (1,0);
    \draw[-,very thick] (1,0) to (1,1);
    \draw[-,very thick] (1,1) to (0,1);
    \draw[-,very thick] (0,1) to (0,0);
\end{tikzpicture}

%% file: bottom_right_tile.tex
\begin{tikzpicture}[scale=1.2]
    \fill[black] (0,0) -- (1,1) -- (1,0) -- cycle;

    \draw[-,very thick] (0,0) to (1,0);
    \draw[-,very thick] (1,0) to (1,1);
    \draw[-,very thick] (1,1) to (0,1);
    \draw[-,very thick] (0,1) to (0,0);
\end{tikzpicture}

%% file: top_right_tile.tex
\begin{tikzpicture}[scale=1.2]
    \fill[black] (1,0) -- (1,1) -- (0,1) -- cycle;

    \draw[-,very thick] (0,0) to (1,0);
    \draw[-,very thick] (1,0) to (1,1);
    \draw[-,very thick] (1,1) to (0,1);
    \draw[-,very thick] (0,1) to (0,0);
\end{tikzpicture}

%% file: top_left_tile.tex
\begin{tikzpicture}[scale=1.2]
    \fill[black] (0,0) -- (1,1) -- (0,1) -- cycle;

    \draw[-,very thick] (0,0) to (1,0);
    \draw[-,very thick] (1,0) to (1,1);
    \draw[-,very thick] (1,1) to (0,1);
    \draw[-,very thick] (0,1) to (0,0);
\end{tikzpicture}

%% file: Deformation.tex
\begin{tikzpicture}
    \tikzset{knoten/.style={circle,fill=gray,inner sep=0.8mm}}

    \draw[-,very thick] (-0.01,0) to (4.01,0);
    \draw[-,very thick, dashed] (0,0) to (2,2);
    \draw[-,very thick, dashed] (2,2) to (4,0);

    \node [knoten,label=below:$v_i$] (a) at (0,0) {};
    \node [knoten,label=below:$v_{i+1}$] (b) at (4,0) {};
    \node [knoten,label=$p$] (c) at (2,2) {};

    \node (d) at (2,-0.4) {$e_i$};
\end{tikzpicture}

%% file: Different_Deformations.tex
\begin{minipage}{0.2\textwidth}
\begin{tikzpicture}
    \draw[-,red!70!black,very thick] (0,0) to (0,1);
    \draw[-,red!70!black,very thick] (0,0) to (1,0);
    \draw[-,cyan,very thick] (1,1) to (0,1);
    \draw[-,cyan,very thick] (1,1) to (1,0);
    \draw [arrows = {-Stealth[length=10pt, inset=2pt]}] (1.5,0.8) -- (2.5,1.5);
    \draw [arrows = {-Stealth[length=10pt, inset=2pt]}] (1.5,0.2) -- (2.5,-0.5);

    \node (a) at (0.5,-0.3) {$e_1$};
    \node (b) at (-0.3,0.5) {$e_2$};
    \node (c) at (1.3,0.5) {$e_4$};
    \node (d) at (0.5,1.25) {$e_3$};

    \draw[-,cyan,very thick,dashed] (3.52,2.5) to (4,2.02);
    \draw[-,red!70!black,very thick,dashed] (3.48,2.5) to (3.98,2);

    \draw[-,red!70!black,very thick] (3,2) to (4,2);
    \draw[-,red!70!black,very thick] (3,2) to (3,3);
    \draw[-,cyan,very thick] (4,3) to (4,2);
    \draw[-,cyan,very thick] (4,3) to (3,3);

    \draw[-,cyan,very thick,dashed] (3.5,3.5) to (4,3);
    \draw[-,cyan,very thick,dashed] (3.5,3.5) to (3,3);
    \draw[-,cyan,very thick,dashed] (3.5,2.5) to (4,3);

    \draw[-,red!70!black,very thick,dashed] (3.5,2.5) to (3,2);
    \draw[-,red!70!black,very thick,dashed] (2.5,2.5) to (3,2);
    \draw[-,red!70!black,very thick,dashed] (2.5,2.5) to (3,3);

    \draw[-,red!70!black,very thick] (3,-2) to (3,-1);
    \draw[-,red!70!black,very thick] (3,-2) to (4,-2);
    \draw[-,cyan,very thick] (4,-1) to (3,-1);
    \draw[-,cyan,very thick] (4,-1) to (4,-2);

    \draw[-,cyan,very thick,dashed] (3.5,-1.5) to (4,-1);
    \draw[-,cyan,very thick,dashed] (3.5,-1.5) to (3,-1);
    \draw[-,cyan,very thick,dashed] (4.5,-1.5) to (4,-1);
    \draw[-,cyan,very thick,dashed] (4.5,-1.5) to (4,-2);

    \draw[-,red!70!black,very thick,dashed] (3.5,-1.5) to (3,-2);
    \draw[-,red!70!black,very thick,dashed] (3.5,-1.5) to (4,-2);
    \draw[-,red!70!black,very thick,dashed] (3,-2) to (2.5,-1.5);
    \draw[-,red!70!black,very thick,dashed] (2.5,-1.5) to (3,-1);
\end{tikzpicture}
\end{minipage}

%% file: top_bottom.tex
\begin{tikzpicture}[scale=1.4]
    \fill[black] (0,0) -- (0.5,0.5) -- (1,0) -- cycle;
    \fill[black] (1,1) -- (0.5,0.5) -- (0,1) -- cycle;

    \draw[-,very thick] (0,0) to (1,0);
    \draw[-,very thick] (1,0) to (1,1);
    \draw[-,very thick] (1,1) to (0,1);
    \draw[-,very thick] (0,1) to (0,0);
\end{tikzpicture}

%% file: left_right.tex
\begin{tikzpicture}[scale=1.4]
    \fill[black] (0,0) -- (0.5,0.5) -- (0,1) -- cycle;
    \fill[black] (1,1) -- (0.5,0.5) -- (1,0) -- cycle;

    \draw[-,very thick] (0,0) to (1,0);
    \draw[-,very thick] (1,0) to (1,1);
    \draw[-,very thick] (1,1) to (0,1);
    \draw[-,very thick] (0,1) to (0,0);
\end{tikzpicture}

%% file: 3colouring.tex
\definecolor{darkgreen}{rgb}{0,0.5,0}
\definecolor{darkblue}{rgb}{0,0,0.8}
\definecolor{darkred}{rgb}{0.8,0,0}

\begin{subfigure}{0.15\textwidth}
\centering
\begin{tikzpicture}[scale=1.5]
    \tikzset{knoten/.style={circle,fill=gray,inner sep=0.8mm}}
    \fill[black] (0,0) -- (1,1) -- (0,1) -- cycle;

    \node [knoten, darkgreen] (a) at (0,0) {};
    \node [knoten, darkblue] (b) at (1,0) {};
    \node [knoten, darkred] (c) at (1,1) {};
    \node [knoten, darkblue] (d) at (0,1) {};

    \draw[-,very thick] (a) to (b);
    \draw[-,very thick] (b) to (c);
    \draw[-,very thick] (c) to (d);
    \draw[-,very thick] (d) to (a);
\end{tikzpicture}
\caption{}
\label{subfig:a}
\end{subfigure}
\begin{subfigure}{0.15\textwidth}
\centering
\begin{tikzpicture}[scale=1.5]
    \tikzset{knoten/.style={circle,fill=gray,inner sep=0.8mm}}
        \fill[black] (0,0) -- (1,1) -- (1,0) -- cycle;
    
        \node [knoten, darkred] (a) at (0,0) {};
        \node [knoten, darkblue] (b) at (1,0) {};
        \node [knoten, darkgreen] (c) at (1,1) {};
        \node [knoten, darkblue] (d) at (0,1) {};
    
        \draw[-,very thick] (a) to (b);
        \draw[-,very thick] (b) to (c);
        \draw[-,very thick] (c) to (d);
        \draw[-,very thick] (d) to (a);
\end{tikzpicture}
\caption{}
\label{subfig:b}
\end{subfigure}
\begin{subfigure}{0.15\textwidth}
\centering
\begin{tikzpicture}[scale=1.5]
    \tikzset{knoten/.style={circle,fill=gray,inner sep=0.8mm}}
        \fill[black] (0,0) -- (1,0) -- (0,1) -- cycle;
    
        \node [knoten, darkgreen] (a) at (0,0) {};
        \node [knoten, darkred] (b) at (1,0) {};
        \node [knoten, darkgreen] (c) at (1,1) {};
        \node [knoten, darkblue] (d) at (0,1) {};
    
        \draw[-,very thick] (a) to (b);
        \draw[-,very thick] (b) to (c);
        \draw[-,very thick] (c) to (d);
        \draw[-,very thick] (d) to (a);
\end{tikzpicture}
\caption{}
\label{subfig:c}
\end{subfigure}
\begin{subfigure}{0.15\textwidth}
\centering
\begin{tikzpicture}[scale=1.5]
    \tikzset{knoten/.style={circle,fill=gray,inner sep=0.8mm}}
        \fill[black] (1,0) -- (1,1) -- (0,1) -- cycle;
    
        \node [knoten, darkred] (a) at (0,0) {};
        \node [knoten, darkgreen] (b) at (1,0) {};
        \node [knoten, darkred] (c) at (1,1) {};
        \node [knoten, darkblue] (d) at (0,1) {};
    
        \draw[-,very thick] (a) to (b);
        \draw[-,very thick] (b) to (c);
        \draw[-,very thick] (c) to (d);
        \draw[-,very thick] (d) to (a);
\end{tikzpicture}
\caption{}
\label{subfig:d}
\end{subfigure}
\begin{subfigure}{0.15\textwidth}
\centering
\begin{tikzpicture}[scale=1.5]
    \tikzset{knoten/.style={circle,fill=gray,inner sep=0.8mm}}

    \fill[black] (0,0) -- (0.5,0.5) -- (0,1) -- cycle;
    \fill[black] (1,1) -- (0.5,0.5) -- (1,0) -- cycle;

    \node [knoten, darkgreen] (a) at (0,0) {};
    \node [knoten, darkblue] (b) at (1,0) {};
    \node [knoten, darkgreen] (c) at (1,1) {};
    \node [knoten, darkblue] (d) at (0,1) {};

    \draw[-,very thick] (a) to (b);
    \draw[-,very thick] (b) to (c);
    \draw[-,very thick] (c) to (d);
    \draw[-,very thick] (d) to (a);

    \draw[-,very thick] (a) to (c);
    \draw[-,very thick] (b) to (d);
\end{tikzpicture}
\caption{}
\label{subfig:e}
\end{subfigure}
\begin{subfigure}{0.15\textwidth}
\centering
\begin{tikzpicture}[scale=1.5]
    \tikzset{knoten/.style={circle,fill=gray,inner sep=0.8mm}}

    \fill[black] (0,0) -- (0.5,0.5) -- (1,0) -- cycle;
    \fill[black] (1,1) -- (0.5,0.5) -- (0,1) -- cycle;

    \node [knoten, darkred] (a) at (0,0) {};
    \node [knoten, darkblue] (b) at (1,0) {};
    \node [knoten, darkred] (c) at (1,1) {};
    \node [knoten, darkblue] (d) at (0,1) {};

    \draw[-,very thick] (a) to (b);
    \draw[-,very thick] (b) to (c);
    \draw[-,very thick] (c) to (d);
    \draw[-,very thick] (d) to (a);

    \draw[-,very thick] (a) to (c);
    \draw[-,very thick] (b) to (d);
\end{tikzpicture}
\caption{}
\label{subfig:f}
\end{subfigure}

%% file: missed_colouring.tex
\definecolor{darkgreen}{rgb}{0,0.5,0}
\definecolor{darkblue}{rgb}{0,0,0.8}
\definecolor{darkred}{rgb}{0.8,0,0}
\begin{tikzpicture}[scale=1.5]
    \tikzset{knoten/.style={circle,fill=gray,inner sep=0.8mm}}

    \fill[black] (0,0) -- (0.5,0.5) -- (1,0) -- cycle;
    \fill[black] (1,1) -- (0.5,0.5) -- (0,1) -- cycle;

    \node [knoten, darkgreen] (a) at (0,0) {};
    \node [knoten, darkred] (b) at (1,0) {};
    \node [knoten, darkgreen] (c) at (1,1) {};
    \node [knoten, darkred] (d) at (0,1) {};

    \draw[-,very thick] (a) to (b);
    \draw[-,very thick] (b) to (c);
    \draw[-,very thick] (c) to (d);
    \draw[-,very thick] (d) to (a);

    \draw[-,very thick] (a) to (c);
    \draw[-,very thick] (b) to (d);
\end{tikzpicture}

%% file: tiling.tex
\definecolor{darkgreen}{rgb}{0,0.5,0}
\definecolor{darkblue}{rgb}{0,0,0.8}
\definecolor{darkred}{rgb}{0.8,0,0}

\begin{minipage}{0.99\textwidth}
\centering
\begin{tikzpicture}[scale=1.5]
    \tikzset{knoten/.style={circle,fill=gray,inner sep=0.8mm}}
    \fill[black] (0,0) -- (0.5,0.5) -- (1,0) -- cycle;
    \fill[black] (1,1) -- (0.5,0.5) -- (0,1) -- cycle;

    \node [knoten, darkblue] (a) at (0,0) {};
    \node [knoten, darkgreen] (b) at (1,0) {};
    \node [knoten, darkblue] (c) at (1,1) {};
    \node [knoten, darkgreen] (d) at (0,1) {};

    \draw[-,very thick] (a) to (b);
    \draw[-,very thick] (b) to (c);
    \draw[-,very thick] (c) to (d);
    \draw[-,very thick] (d) to (a);
\end{tikzpicture}
\hspace{0.2cm}
\begin{tikzpicture}[scale=1.5]
    \tikzset{knoten/.style={circle,fill=gray,inner sep=0.8mm}}
        \fill[black] (0,0) -- (0,1) -- (1,0) -- cycle;
    
        \node [knoten, darkblue] (a) at (0,0) {};
        \node [knoten, darkgreen] (b) at (1,0) {};
        \node [knoten, darkblue] (c) at (1,1) {};
        \node [knoten, darkred] (d) at (0,1) {};
    
        \draw[-,very thick] (a) to (b);
        \draw[-,very thick] (b) to (c);
        \draw[-,very thick] (c) to (d);
        \draw[-,very thick] (d) to (a);
\end{tikzpicture}
\hspace{0.2cm}
\begin{tikzpicture}[scale=1.5]
    \tikzset{knoten/.style={circle,fill=gray,inner sep=0.8mm}}
        \fill[black] (0,0) -- (1,0) -- (1,1) -- cycle;
    
        \node [knoten, darkblue] (a) at (0,0) {};
        \node [knoten, darkgreen] (b) at (1,0) {};
        \node [knoten, darkred] (c) at (1,1) {};
        \node [knoten, darkgreen] (d) at (0,1) {};
    
        \draw[-,very thick] (a) to (b);
        \draw[-,very thick] (b) to (c);
        \draw[-,very thick] (c) to (d);
        \draw[-,very thick] (d) to (a);
\end{tikzpicture}
\hspace{1.5cm}
\begin{tikzpicture}[scale=1.5]
    \tikzset{knoten/.style={circle,fill=gray,inner sep=0.8mm}}
        \fill[black] (0,0) -- (0.5,0.5) -- (0,1) -- cycle;
        \fill[black] (1,1) -- (0.5,0.5) -- (1,0) -- cycle;
    
        \node [knoten, darkgreen] (a) at (0,0) {};
        \node [knoten, darkblue] (b) at (1,0) {};
        \node [knoten, darkgreen] (c) at (1,1) {};
        \node [knoten, darkblue] (d) at (0,1) {};
    
        \draw[-,very thick] (a) to (b);
        \draw[-,very thick] (b) to (c);
        \draw[-,very thick] (c) to (d);
        \draw[-,very thick] (d) to (a);
\end{tikzpicture}
\hspace{0.2cm}
\begin{tikzpicture}[scale=1.5]
    \tikzset{knoten/.style={circle,fill=gray,inner sep=0.8mm}}

    \fill[black] (0,0) -- (1,0) -- (1,1) -- cycle;

    \node [knoten, darkred] (a) at (0,0) {};
    \node [knoten, darkblue] (b) at (1,0) {};
    \node [knoten, darkgreen] (c) at (1,1) {};
    \node [knoten, darkblue] (d) at (0,1) {};

    \draw[-,very thick] (a) to (b);
    \draw[-,very thick] (b) to (c);
    \draw[-,very thick] (c) to (d);
    \draw[-,very thick] (d) to (a);
\end{tikzpicture}
\hspace{0.2cm}
\begin{tikzpicture}[scale=1.5]
    \tikzset{knoten/.style={circle,fill=gray,inner sep=0.8mm}}

    \fill[black] (0,0) -- (0,1) -- (1,0) -- cycle;

    \node [knoten, darkgreen] (a) at (0,0) {};
    \node [knoten, darkred] (b) at (1,0) {};
    \node [knoten, darkgreen] (c) at (1,1) {};
    \node [knoten, darkblue] (d) at (0,1) {};

    \draw[-,very thick] (a) to (b);
    \draw[-,very thick] (b) to (c);
    \draw[-,very thick] (c) to (d);
    \draw[-,very thick] (d) to (a);
\end{tikzpicture}
\end{minipage}

%% file: Example.tex
\definecolor{darkgreen}{rgb}{0,0.5,0}
\definecolor{darkblue}{rgb}{0,0,0.8}
\definecolor{darkred}{rgb}{0.8,0,0}

\begin{tikzpicture}[scale=1.2]
    \tikzset{knoten/.style={circle,fill=gray,inner sep=0.8mm}}
    \fill[black] (0,0) -- (1,1) -- (1,0) -- cycle;
    \fill[black] (2,0) -- (2,1) -- (1,1) -- cycle;
    \fill[black] (2,0) -- (2.5,0.5) -- (3,0) -- cycle;
    \fill[black] (2,1) -- (2.5,0.5) -- (3,1) -- cycle;

    \fill[black] (0,1) -- (0.5,1.5) -- (1,1) -- cycle;
    \fill[black] (0,2) -- (0.5,1.5) -- (1,2) -- cycle;
    \fill[black] (1,1) -- (1,2) -- (2,2) -- cycle;
    \fill[black] (2,2) -- (3,2) -- (2,1) -- cycle;

    \fill[black] (0,2) -- (1,3) -- (0,3) -- cycle;
    \fill[black] (1,2) -- (1.5,2.5) -- (1,3) -- cycle;
    \fill[black] (2,3) -- (1.5,2.5) -- (2,2) -- cycle;
    \fill[black] (2,3) -- (3,3) -- (3,2) -- cycle;

    \node [knoten, darkred] (a0) at (0,0) {};
    \node [knoten, darkblue] (a1) at (1,0) {};
    \node [knoten, darkred] (a2) at (2,0) {};
    \node [knoten, darkblue] (a3) at (3,0) {};

    \node [knoten, darkblue] (b0) at (0,1) {};
    \node [knoten, darkgreen] (b1) at (1,1) {};
    \node [knoten, darkblue] (b2) at (2,1) {};
    \node [knoten, darkred] (b3) at (3,1) {};

    \node [knoten, darkgreen] (c0) at (0,2) {};
    \node [knoten, darkblue] (c1) at (1,2) {};
    \node [knoten, darkred] (c2) at (2,2) {};
    \node [knoten, darkgreen] (c3) at (3,2) {};

    \node [knoten, darkblue] (d0) at (0,3) {};
    \node [knoten, darkred] (d1) at (1,3) {};
    \node [knoten, darkblue] (d2) at (2,3) {};
    \node [knoten, darkred] (d3) at (3,3) {};

    \draw[-,very thick] (a0) to (a1);
    \draw[-,very thick] (a1) to (a2);
    \draw[-,very thick] (a2) to (a3);

    \draw[-,very thick] (b0) to (b1);
    \draw[-,very thick] (b1) to (b2);
    \draw[-,very thick] (b2) to (b3);

    \draw[-,very thick] (c0) to (c1);
    \draw[-,very thick] (c1) to (c2);
    \draw[-,very thick] (c2) to (c3);

    \draw[-,very thick] (d0) to (d1);
    \draw[-,very thick] (d1) to (d2);
    \draw[-,very thick] (d2) to (d3);

    \draw[-,very thick] (a0) to (b0);
    \draw[-,very thick] (b0) to (c0);
    \draw[-,very thick] (c0) to (d0);

    \draw[-,very thick] (a1) to (b1);
    \draw[-,very thick] (b1) to (c1);
    \draw[-,very thick] (c1) to (d1);

    \draw[-,very thick] (a2) to (b2);
    \draw[-,very thick] (b2) to (c2);
    \draw[-,very thick] (c2) to (d2);

    \draw[-,very thick] (a3) to (b3);
    \draw[-,very thick] (b3) to (c3);
    \draw[-,very thick] (c3) to (d3);
\end{tikzpicture}

%% file: ZigZagAssembly.tex
\begin{tikzpicture}[scale=1.4]
    \fill[black] (0,1) -- (0.5,0.5) -- (0.5,1) -- cycle;
    \fill[black] (1,0) -- (0.5,0.5) -- (1,0.5) -- cycle;
    \fill[black] (0.5,0) -- (0,0) -- (0,0.5) -- (0.5,0.5) -- cycle;

    \begin{scope}[shift={(1,1)},rotate=270]
       \fill[black] (0,1) -- (0.5,0.5) -- (0.5,1) -- cycle;
        \fill[black] (1,0) -- (0.5,0.5) -- (1,0.5) -- cycle;
        \fill[black] (0.5,0) -- (0,0) -- (0,0.5) -- (0.5,0.5) -- cycle;
    \end{scope}

    \begin{scope}[shift={(2,0)}]
        \fill[black] (0,1) -- (0.5,0.5) -- (0.5,1) -- cycle;
    \fill[black] (1,0) -- (0.5,0.5) -- (1,0.5) -- cycle;
    \fill[black] (0.5,0) -- (0,0) -- (0,0.5) -- (0.5,0.5) -- cycle;
    \end{scope}

    \begin{scope}[shift={(0,-1)}]
        \fill[black] (0,0) -- (0.5,0) -- (0.5,0.5) -- cycle;
        \fill[black] (0,1) -- (0,0.5) -- (0.5,0.5) -- cycle;
        \fill[black] (1,1) -- (0.5,1) -- (0.5,0.5) -- cycle;
        \fill[black] (1,0) -- (1,0.5) -- (0.5,0.5) -- cycle;
    \end{scope}
    \begin{scope}[shift={(1,-1)}]
        \fill[black] (0,0) -- (0.5,0) -- (0.5,0.5) -- cycle;
        \fill[black] (0,1) -- (0,0.5) -- (0.5,0.5) -- cycle;
        \fill[black] (1,1) -- (0.5,1) -- (0.5,0.5) -- cycle;
        \fill[black] (1,0) -- (1,0.5) -- (0.5,0.5) -- cycle;
    \end{scope}
    \begin{scope}[shift={(2,-1)}]
        \fill[black] (0,0) -- (0.5,0) -- (0.5,0.5) -- cycle;
        \fill[black] (0,1) -- (0,0.5) -- (0.5,0.5) -- cycle;
        \fill[black] (1,1) -- (0.5,1) -- (0.5,0.5) -- cycle;
        \fill[black] (1,0) -- (1,0.5) -- (0.5,0.5) -- cycle;
    \end{scope}

    \begin{scope}[shift={(1,-2)},rotate=90]
        \fill[black] (0,0) -- (0.5,0.5) -- (0,0.5) -- cycle;
        \fill[black] (1,1) -- (0.5,0.5) -- (0.5,1) -- cycle;
        \fill[black] (0.5,0) -- (1,0) -- (1,0.5) -- (0.5,0.5) -- cycle;
    \end{scope}

    \begin{scope}[shift={(1,-2)}]
        \fill[black] (0,0) -- (0.5,0.5) -- (0,0.5) -- cycle;
        \fill[black] (1,1) -- (0.5,0.5) -- (0.5,1) -- cycle;
        \fill[black] (0.5,0) -- (1,0) -- (1,0.5) -- (0.5,0.5) -- cycle;
    \end{scope}

    \begin{scope}[shift={(3,-2)},rotate=90]
        \fill[black] (0,0) -- (0.5,0.5) -- (0,0.5) -- cycle;
        \fill[black] (1,1) -- (0.5,0.5) -- (0.5,1) -- cycle;
        \fill[black] (0.5,0) -- (1,0) -- (1,0.5) -- (0.5,0.5) -- cycle;
    \end{scope}

    \draw[-,very thick] (0,1) to (3,1);
    \draw[-,very thick] (0,1) to (0,-2);
    \draw[-,very thick] (3,1) to (3,-2);
    \draw[-,very thick] (0,-2) to (3,-2);
\end{tikzpicture}

%% file: Different_Deformations_Lozenge.tex
\begin{minipage}{0.2\textwidth}
\begin{tikzpicture}[scale=1.2]
    \begin{scope}[rotate=60]
    \draw[-,red!70!black,very thick] (0,0) to (0.5,0.866);
    \draw[-,cyan,very thick] (0.5,0.866) to (1,0);
    \draw[-,cyan,very thick] (1,0) to (0.5,-0.866);
    \draw[-,red!70!black,very thick] (0.5,-0.866) to (0,0);
    \end{scope}
    
    \draw [arrows = {-Stealth[length=10pt, inset=2pt]}] (1.5,0.8) -- (2.5,1.5);
    \draw [arrows = {-Stealth[length=10pt, inset=2pt]}] (1.5,0.2) -- (2.5,-0.5);

    \begin{scope}[shift={(0.5,-0.5)}]
    \begin{scope}[rotate around={60:(3,2)}]
    \draw[-,red!70!black,very thick] (3,2) to (3.5,2.866);
    \draw[-,cyan,very thick] (3.5,2.866) to (4,2);
    \draw[-,cyan,very thick] (4,2) to (3.5,-0.866+2);
    \draw[-,red!70!black,very thick] (3.5,-0.866+2) to (3,2);
    
    \draw[-,red!70!black,very thick,dashed] (3,2) to (3.5,2+0.866/3);
    \draw[-,red!70!black,very thick,dashed] (3.48,2+0.866/3) to (3.48,2.866);
    \draw[-,red!70!black,very thick,dashed] (3.5,-0.866+2) to (3,-2*0.866/3+2);
    \draw[-,red!70!black,very thick,dashed] (3,-2*0.866/3+2) to (3,2);

    \draw[-,cyan,very thick,dashed] (3.5,2+0.866/3) to (4,2);
    \draw[-,cyan,very thick,dashed] (3.52,2+0.866/3) to (3.52,2.866);
    \draw[-,cyan,very thick,dashed] (4,2) to (4,-2*0.866/3+2);
    \draw[-,cyan,very thick,dashed] (4,-2*0.866/3+2) to (3.5,-0.866+2);
    \end{scope}
    \end{scope}
    
    \begin{scope}[shift={(0.5,0.5)}]
    \begin{scope}[rotate around={60:(3,-2)}]
    \draw[-,red!70!black,very thick] (3,-2) to (3.5,0.866-2);
    \draw[-,cyan,very thick] (3.5,0.866-2) to (4,-2);
    \draw[-,cyan,very thick] (4,-2) to (3.5,-2.866);
    \draw[-,red!70!black,very thick] (3.5,-2.866) to (3,-2);

    \draw[-,cyan,very thick,dashed] (3.5,-2+0.866/3) to (4,-2);
    \draw[-,cyan,very thick,dashed] (3.5,-2+0.866/3) to (3.5,0.866-2);
    \draw[-,cyan,very thick,dashed] (4,-2) to (4,-2*0.866/3-2);
    \draw[-,cyan,very thick,dashed] (4,-2*0.866/3-2) to (3.5,-0.866-2);

    \draw[-,red!70!black,very thick,dashed] (3,-2) to (3,-2+2*0.866/3);
    \draw[-,red!70!black,very thick,dashed] (3,-2+2*0.866/3) to (3.5,0.866-2);
    
    \draw[-,red!70!black,very thick,dashed] (3.5,-0.866-2) to (3.5,-0.866/3-2);
    \draw[-,red!70!black,very thick,dashed] (3.5,-0.866/3-2) to (3,-2);
    \end{scope}
    \end{scope}
\end{tikzpicture}
\end{minipage}

%% file: Lozenge_Rhom.tex
\begin{tikzpicture}[scale=1.5]
    \tikzset{knoten/.style={circle,fill=gray,inner sep=0.8mm}}

    \begin{scope}[rotate=60]
        \fill[black] (0,0) -- (1,0) -- (0.5,0.866) -- cycle;
    
        \draw[-,very thick] (0,0) to (0.5,0.866);
        \draw[-,very thick] (0.5,0.866) to (1,0);
        \draw[-,very thick] (1,0) to (0.5,-0.866);
        \draw[-,very thick] (0.5,-0.866) to (0,0);
    \end{scope}
\end{tikzpicture}

%% file: Lozenge_p3.tex
\begin{tikzpicture}[scale=1.5]
    \tikzset{knoten/.style={circle,fill=gray,inner sep=0.8mm}}

    \begin{scope}[rotate=60]
        \fill[black] (0.5,0) -- (0,0) -- (0.5,0.866) -- cycle;
        \fill[black] (0.5,0) -- (1,0) -- (0.5,-0.866) -- cycle;
    
        \draw[-,very thick] (0,0) to (0.5,0.866);
        \draw[-,very thick] (0.5,0.866) to (1,0);
        \draw[-,very thick] (1,0) to (0.5,-0.866);
        \draw[-,very thick] (0.5,-0.866) to (0,0);
    \end{scope}
\end{tikzpicture}

%% file: LozengeTiling1.tex
\begin{tikzpicture}[scale=1.2]
    \tikzset{knoten/.style={circle,fill=gray,inner sep=0.8mm}}

    \draw[very thick] (-0.5,0.866) -- (0,1.732);
    \draw[very thick] (0,1.732) -- (0.5,0.866);
    \draw[very thick] (0.5,0.866) -- (0,0);
    \draw[very thick] (-0.5,0.866) -- (0,0);
    
    \fill[black] (0,1.732) -- (-0.5,0.866) -- (0.5,0.866) -- cycle;
    
    \draw[very thick] (-1.0,0.0) -- (-1.5,0.866);
    \draw[very thick] (-1.5,0.866) -- (-0.5,0.866);
    \draw[very thick] (-0.5,0.866) -- (0,0);
    \draw[very thick] (-1.0,0.0) -- (0,0);

    \fill[black] (0,0) -- (-0.5,0.866) -- (-0.75,0.433) -- cycle;
    \fill[black] (-1,0) -- (-1.5,0.866) -- (-0.75,0.433) -- cycle;
    
    \draw[very thick] (-0.5,-0.866) -- (-1.5,-0.866);
    \draw[very thick] (-1.5,-0.866) -- (-1.0,0.0);
    \draw[very thick] (-1.0,0.0) -- (0,0);
    \draw[very thick] (-0.5,-0.866) -- (0,0);

    \fill[black] (0,0) -- (-1,0) -- (-0.75,-0.433) -- cycle;
    \fill[black] (-0.5,-0.866) -- (-1.5,-0.866) -- (-0.75,-0.433) -- cycle;
    
    \draw[very thick] (0.5,-0.866) -- (0,-1.732);
    \draw[very thick] (0,-1.732) -- (-0.5,-0.866);
    \draw[very thick] (-0.5,-0.866) -- (0,0);
    \draw[very thick] (0.5,-0.866) -- (0,0);

    \fill[black] (0,0) -- (0,-0.866) -- (-0.5,-0.866) -- cycle;
    \fill[black] (0,-1.732) -- (0,-0.866) -- (0.5,-0.866) -- cycle;
    
    \draw[very thick] (1.0,0.0) -- (1.5,-0.866);
    \draw[very thick] (1.5,-0.866) -- (0.5,-0.866);
    \draw[very thick] (0.5,-0.866) -- (0,0);
    \draw[very thick] (1.0,0.0) -- (0,0);

    \fill[black] (0,0) -- (0.5,-0.866) -- (0.75,-0.433) -- cycle;
    \fill[black] (1,0) -- (1.5,-0.866) -- (0.75,-0.433) -- cycle;
    
    \draw[very thick] (0.5,0.866) -- (1.5,0.866);
    \draw[very thick] (1.5,0.866) -- (1.0,0.0);
    \draw[very thick] (1.0,0.0) -- (0,0);
    \draw[very thick] (0.5,0.866) -- (0,0);

    \fill[black] (0,0) -- (1,0) -- (0.5,0.866) -- cycle;

    \draw[-,very thick] (1,0) to (1.5,0.866);
    \draw[-,very thick] (1.5,0.866) to (2,0);
    \draw[-,very thick] (2,0) to (1.5,-0.866);
    \draw[-,very thick] (1,0) to (1.5,-0.866);

    \fill[black] (2,0) -- (1.5,0) -- (1.5,-0.866) -- cycle;
    \fill[black] (1,0) -- (1.5,0) -- (1.5,0.866) -- cycle;

    \draw[-,very thick] (-2,0) to (-1.5,0.866);
    \draw[-,very thick] (-1.5,0.866) to (-1,0);
    \draw[-,very thick] (-1,0) to (-1.5,-0.866);
    \draw[-,very thick] (-2,0) to (-1.5,-0.866);

    \fill[black] (-1,0) -- (-2,0) -- (-1.5,-0.866) -- cycle;

    \draw[very thick] (-1,-2*0.866) -- (-1.5,-0.866);
    \draw[very thick] (-1.5,-0.866) -- (-0.5,-0.866);
    \draw[very thick] (-0.5,-0.866) -- (0,-2*0.866);
    \draw[very thick] (-1,-2*0.866) -- (0,-2*0.866);

    \fill[black] (0,-2*0.866) -- (-0.5,-0.866) -- (-1,-2*0.866) -- cycle;

    \draw[very thick] (-0.5,0.866) -- (-1.5,0.866);
    \draw[very thick] (-1.5,0.866) -- (-1.0,2*0.866);
    \draw[very thick] (-1,2*0.866) -- (0,2*0.866);
    \draw[very thick] (-0.5,0.866) -- (0,2*0.866);

    \fill[black] (-1.5,0.866) -- (-1,2*0.866) -- (-0.5,0.866) -- cycle;

    \draw[very thick] (1,2*0.866) -- (1.5,0.866);
    \draw[very thick] (1.5,0.866) -- (0.5,0.866);
    \draw[very thick] (0.5,0.866) -- (0,2*0.866);
    \draw[very thick] (1,2*0.866) -- (0,2*0.8660);

    \fill[black] (0.5,0.866) -- (1.5,0.866) -- (1,2*0.866) -- cycle;

    \draw[very thick] (0.5,0.866-2*0.866) -- (1.5,0.866-2*0.866);
    \draw[very thick] (1.5,0.866-2*0.866) -- (1,-2*0.866);
    \draw[very thick] (1,-2*0.866) -- (0,-2*0.866);
    \draw[very thick] (0.5,-0.866) -- (0,-2*0.866);

    \fill[black] (1.5,0.866-2*0.866) -- (1,-2*0.866) -- (0.5,-0.866) -- cycle;
\end{tikzpicture}

%% file: Different_Deformations_Kite.tex
\begin{minipage}{0.2\textwidth}
\begin{tikzpicture}
    \draw[-,red!70!black,very thick] (0,0) to (-0.866,1.5);
    \draw[-,cyan,very thick] (-0.866,1.5) to (0,2);
    \draw[-,cyan,very thick] (0,2) to (0.866,1.5);
    \draw[-,red!70!black,very thick] (0.866,1.5) to (0,0);
    
    \draw [arrows = {-Stealth[length=10pt, inset=2pt]}] (1.5,0.8) -- (2.5,1.5);
    \draw [arrows = {-Stealth[length=10pt, inset=2pt]}] (1.5,0.2) -- (2.5,-0.5);

    \begin{scope}[shift={(4,1.5)}]
        \draw[-,red!70!black,very thick] (0,0) to (-0.866,1.5);
        \draw[-,cyan,very thick] (-0.866,1.5) to (0,2);
        \draw[-,cyan,very thick] (0,2) to (0.866,1.5);
        \draw[-,red!70!black,very thick] (0.866,1.5) to (0,0);

        \draw[-,red!70!black,very thick,dashed] (0.866-0.02,1.5-0.02) to (0-0.02,1-0.02);
        \draw[-,red!70!black,very thick,dashed] (0,0) to (0,1);
        \draw[-,red!70!black,very thick,dashed] (0,0) to (-0.866,0.5);
        \draw[-,red!70!black,very thick,dashed] (-0.866,1.5) to (-0.866,0.5);
    
        \draw[-,cyan,very thick,dashed] (0.866+0.02,1.5+0.02) to (0+0.02,1+0.02);
        \draw[-,cyan,very thick,dashed] (0,2) to (0,1);
        
        \draw[-,cyan,very thick,dashed] (0,2) to (-0.866,2.5);
        \draw[-,cyan,very thick,dashed] (-0.866,1.5) to (-0.866,2.5);
    \end{scope}
    
    \begin{scope}[shift={(4,-1.5)}]
        \draw[-,red!70!black,very thick] (0,0) to (-0.866,1.5);
        \draw[-,cyan,very thick] (-0.866,1.5) to (0,2);
        \draw[-,cyan,very thick] (0,2) to (0.866,1.5);
        \draw[-,red!70!black,very thick] (0.866,1.5) to (0,0);

        \draw[-,red!70!black,very thick,dashed] (0.866,1.5) to (0,1);
        \draw[-,red!70!black,very thick,dashed] (0,0) to (0,1);
        \draw[-,red!70!black,very thick,dashed] (0,0) to (-0.866,0.5);
        \draw[-,red!70!black,very thick,dashed] (-0.866,1.5) to (-0.866,0.5);
    
        \draw[-,cyan,very thick,dashed] (-0.866,1.5) to (0,1);
        \draw[-,cyan,very thick,dashed] (0,2) to (0,1);
        
        \draw[-,cyan,very thick,dashed] (0,2) to (0.866,2.5);
        \draw[-,cyan,very thick,dashed] (0.866,1.5) to (0.866,2.5);
    \end{scope}
\end{tikzpicture}
\end{minipage}

%% file: versatile_p6_tile.tex
\begin{tikzpicture}
    \tikzset{knoten/.style={circle,fill=gray,inner sep=0.8mm}}
    \begin{scope}[rotate=270]
        \draw[-,very thick] (0,0) to (-0.866,1.5);
        \draw[-,very thick] (-0.866,1.5) to (0,2);
        \draw[-,very thick] (0,2) to (0.866,1.5);
        \draw[-,very thick] (0.866,1.5) to (0,0);
    
        \fill[black] (0,0) -- (0,2) -- (0.866,1.5) -- cycle;
    \end{scope}
\end{tikzpicture}

%% file: Bilozenge_tile.tex
\begin{tikzpicture}
    \tikzset{knoten/.style={circle,fill=gray,inner sep=0.8mm}}
    \begin{scope}[rotate=270]
        \draw[-,very thick] (0,0) to (-0.866,1.5);
        \draw[-,very thick] (-0.866,1.5) to (0,2);
        \draw[-,very thick] (0,2) to (0.866,1.5);
        \draw[-,very thick] (0.866,1.5) to (0,0);
    
        \fill[black] (0,0) -- (0,1) -- (0.866,1.5) -- cycle;
        \fill[black] (0,2) -- (0,1) -- (-0.866,1.5) -- cycle;
    \end{scope}
\end{tikzpicture}

%% file: Different_Deformations_p1.tex
\begin{minipage}{0.2\textwidth}
\begin{tikzpicture}
    \begin{scope}[scale=1.2]
        \draw[-,red!70!black,very thick] (0,0) to (0,1);
        \draw[-,cyan,very thick] (0,0) to (1,0);
        \draw[-,cyan,very thick] (1,1) to (0,1);
        \draw[-,red!70!black,very thick] (1,1) to (1,0);
    \end{scope}

    \draw [arrows = {-Stealth[length=10pt, inset=2pt]}] (1.5,0.8) -- (2.5,1.5);
    \draw [arrows = {-Stealth[length=10pt, inset=2pt]}] (1.5,0.2) -- (2.5,-0.5);

    \begin{scope}[scale=1.5,shift={(-1,-1)}]
    \draw[-,cyan,very thick] (3,2) to (4,2);
    \draw[-,red!70!black,very thick] (3,2) to (3,3);
    \draw[-,red!70!black,very thick] (4,3) to (4,2);
    \draw[-,cyan,very thick] (4,3) to (3,3);

    \draw[-,cyan,very thick,dashed] (3.25+0.02,2.25) to (3+0.02,2);
    \draw[-,cyan,very thick,dashed] (3.25,2.25) to (3.75,1.75);
    \draw[-,cyan,very thick,dashed] (4,2) to (3.75,1.75);
    \draw[-,cyan,very thick,dashed] (3.25,3.25) to (3,3);
    \draw[-,cyan,very thick,dashed] (3.25,3.25) to (3.75,2.75);
    \draw[-,cyan,very thick,dashed] (4-0.02,3) to (3.75-0.02,2.75);

    \draw[-,red!70!black,very thick,dashed] (3.25,2.25+0.02) to (3,2+0.02);
    \draw[-,red!70!black,very thick,dashed] (3.25,2.25) to (2.75,2.75);
    \draw[-,red!70!black,very thick,dashed] (3,3) to (2.75,2.75);
    \draw[-,red!70!black,very thick,dashed] (4.25,2.25) to (4,2);
    \draw[-,red!70!black,very thick,dashed] (4.25,2.25) to (3.75,2.75);
    \draw[-,red!70!black,very thick,dashed] (4,3-0.02) to (3.75,2.75-0.02);
    \end{scope}

    \begin{scope}[scale=1.5,shift={(-1,1)}]
    \draw[-,red!70!black,very thick] (3,-2) to (3,-1);
    \draw[-,cyan,very thick] (3,-2) to (4,-2);
    \draw[-,cyan,very thick] (4,-1) to (3,-1);
    \draw[-,red!70!black,very thick] (4,-1) to (4,-2);

    \draw[-,cyan,very thick,dashed] (3.25,2.25-4) to (3,2-4);
    \draw[-,cyan,very thick,dashed] (3.25,2.25-4) to (3.75,1.75-4);
    \draw[-,cyan,very thick,dashed] (4,2-4) to (3.75,1.75-4);
    \draw[-,cyan,very thick,dashed] (3.25,3.25-4) to (3,3-4);
    \draw[-,cyan,very thick,dashed] (3.25,3.25-4) to (3.75,2.75-4);
    \draw[-,cyan,very thick,dashed] (4,3-4) to (3.75,2.75-4);

    \draw[-,red!70!black,very thick,dashed] (2.75,2.25-4) to (3,2-4);
    \draw[-,red!70!black,very thick,dashed] (2.75,2.25-4) to (3.25,2.75-4);
    \draw[-,red!70!black,very thick,dashed] (3,3-4) to (3.25,2.75-4);
    \draw[-,red!70!black,very thick,dashed] (3.75,2.25-4) to (4,2-4);
    \draw[-,red!70!black,very thick,dashed] (3.75,2.25-4) to (4.25,2.75-4);
    \draw[-,red!70!black,very thick,dashed] (4,3-4) to (4.25,2.75-4);
    \end{scope}
\end{tikzpicture}
\end{minipage}

%% file: pg_deformation.tex
\begin{minipage}{0.2\textwidth}
\begin{tikzpicture}[scale=1.4]
        \draw[-,red!70!black,very thick] (0,0) to (0,1);
        \draw[-,cyan,very thick] (0,0) to (1,0);
        \draw[-,cyan,very thick] (1,1) to (0,1);
        \draw[-,red!70!black,very thick] (1,1) to (1,0);

    \draw [arrows = {-Stealth[length=10pt, inset=2pt]}] (1.5,0.5) -- (2.,0.5);

    \begin{scope}[shift={(-0.5,-2)}]
    \draw[-,cyan,very thick] (3,2) to (4,2);
    \draw[-,red!70!black,very thick] (3,2) to (3,3);
    \draw[-,red!70!black,very thick] (4,3) to (4,2);
    \draw[-,cyan,very thick] (4,3) to (3,3);

    \draw[-,cyan,very thick,dashed] (3.25,2.25) to (3,2);
    \draw[-,cyan,very thick,dashed] (3.25,2.25) to (3.75,1.75);
    \draw[-,cyan,very thick,dashed] (4,2) to (3.75,1.75);
    \draw[-,cyan,very thick,dashed] (3.25+0.02,2.75) to (3+0.02,3);
    \draw[-,cyan,very thick,dashed] (3.25,2.75) to (3.75,3.25);
    \draw[-,cyan,very thick,dashed] (4,3) to (3.75,3.25);

    \draw[-,red!70!black,very thick,dashed] (2.75,2.25) to (3,2);
    \draw[-,red!70!black,very thick,dashed] (2.75,2.25) to (3.25,2.75);
    \draw[-,red!70!black,very thick,dashed] (3,3-0.02) to (3.25,2.75-0.02);
    \draw[-,red!70!black,very thick,dashed] (3.75,2.25) to (4,2);
    \draw[-,red!70!black,very thick,dashed] (3.75,2.25) to (4.25,2.75);
    \draw[-,red!70!black,very thick,dashed] (4,3) to (4.25,2.75);
    \end{scope}
    
\end{tikzpicture}
\end{minipage}

%% file: p2gg_deformation.tex
\begin{minipage}{0.2\textwidth}
\begin{tikzpicture}[scale=1.3]
        \draw[-,red!70!black,very thick] (0,0) to (0,1);
        \draw[-,cyan,very thick] (0,0) to (1,0);
        \draw[-,cyan,very thick] (1,1) to (0,1);
        \draw[-,red!70!black,very thick] (1,1) to (1,0);

    \draw [arrows = {-Stealth[length=10pt, inset=2pt]}] (1.5,0.5) -- (2.,0.5);

    \begin{scope}[shift={(-0.6,-2)}]
    \draw[-,cyan,very thick] (3,2) to (4,2);
    \draw[-,red!70!black,very thick] (3,2) to (3,3);
    \draw[-,red!70!black,very thick] (4,3) to (4,2);
    \draw[-,cyan,very thick] (4,3) to (3,3);

    \draw[-,cyan,very thick,dashed] (3.5,3.25) to (4,3);
    \draw[-,cyan,very thick,dashed] (3.5,3.25) to (3,3);
    \draw[-,cyan,very thick,dashed] (3.52,2.25) to (4,2.02);
    \draw[-,cyan,very thick,dashed] (3.5,2.25) to (3,2);

    \draw[-,red!70!black,very thick,dashed] (4.25,2) to (4,2);
    \draw[-,red!70!black,very thick,dashed] (4.25,2) to (4,3);
    \draw[-,red!70!black,very thick,dashed] (3.25,3) to (3,2);
    \draw[-,red!70!black,very thick,dashed] (3.25,3) to (3,3);
    \end{scope}
\end{tikzpicture}
\end{minipage}

%% file: DeformationRhom.tex
\begin{minipage}{0.2\textwidth}
    \begin{tikzpicture}[scale=0.4]
        \draw[-,very thick] (-2.732,0.732) to (0.732,2.732);
        \draw[-,very thick] (0.732,2.732) to (2.732,-0.732);
        \draw[-,very thick] (2.732,-0.732) to (-0.732,-2.732);
        \draw[-,very thick] (-0.732,-2.732) to (-2.732,0.732);

        \begin{scope}[shift={(-2.732,0.732+2)}]
            \draw[-,very thick] (-1.732*2,0) to (0,2);
            \draw[-,very thick] (0,2) to (1.732*2,0);
            \draw[-,very thick] (1.732*2,0) to (0,-2);
            \draw[-,very thick] (-1.732*2,0) to (0,-2);

            \draw[-,cyan,very thick,dashed] (1.155-0.03,0) to (-0.03,-2);
            \draw[-,cyan,very thick,dashed] (1.155,0) to (0,2);
            \draw[-,cyan,very thick,dashed] (1.155,0.03) to (1.732*2,0.03);

            \draw[-,cyan,very thick,dashed] (-2.309,-2) to (0,-2);
            \draw[-,cyan,very thick,dashed] (-2.309,2) to (0,2);
            \draw[-,cyan,very thick,dashed] (-2.309,-2) to (-1.732*2,0);
            \draw[-,cyan,very thick,dashed] (-2.309,2) to (-1.732*2,0);

            \draw[-,red!70!black,very thick,dashed] (1.155,-0.03) to (2*1.732,-0.03);
            \draw[-,red!70!black,very thick,dashed] (1.155+0.03,0) to (0.03,-2);

            \begin{scope}[shift={(2,-2)},rotate=270]
                \draw[-,red!70!black,very thick,dashed] (1.155,0) to (2*1.732,0);
                \draw[-,red!70!black,very thick,dashed] (1.155,0) to (0,-2);
            \end{scope}

            \begin{scope}[shift={(2,-2)},rotate=270,shift={(-2,2*1.732)}]
                \draw[-,red!70!black,very thick,dashed] (1.155,0) to (2*1.732,0);
                \draw[-,red!70!black,very thick,dashed] (1.155,0) to (0,-2);
            \end{scope}

            \begin{scope}[shift={(2,-2*1.732)}]
                \draw[-,red!70!black,very thick,dashed] (1.155,0) to (2*1.732,0);
                \draw[-,red!70!black,very thick,dashed] (1.155,0) to (0,-2);
            \end{scope}

        \end{scope}
    \end{tikzpicture}
\end{minipage}

%% file: DeformationVersatile.tex
\begin{minipage}{0.2\textwidth}
    \begin{tikzpicture}[scale=0.4]
        \draw[-,very thick] (-2.732,0.732) to (0.732,2.732);
        \draw[-,very thick] (0.732,2.732) to (2.732,-0.732);
        \draw[-,very thick] (2.732,-0.732) to (-0.732,-2.732);
        \draw[-,very thick] (-0.732,-2.732) to (-2.732,0.732);

        \begin{scope}[shift={(-2.732,0.732+2)}]
            \draw[-,very thick] (-1.732*2,0) to (0,2);
            \draw[-,very thick] (0,2) to (1.732*2,0);
            \draw[-,very thick] (1.732*2,0) to (0,-2);
            \draw[-,very thick] (-1.732*2,0) to (0,-2);

            \draw[-,cyan,very thick,dashed] (0.732-0.02,0.732+0.02) to (-0.05,-2+0.02);
            \draw[-,cyan,very thick,dashed] (0.732,0.732+0.02) to (1.732*2,0.02);

            \begin{scope}[shift={(0,2)},rotate=120,shift={(-3.5,3.9)}]
                \draw[-,cyan,very thick,dashed] (0.732,0.732) to (0,-2);
                \draw[-,cyan,very thick,dashed] (0.732,0.732) to (1.732*2,0);
            \end{scope}

            \begin{scope}[shift={(0,2)},rotate=180,shift={(0,2)}]
                \draw[-,cyan,very thick,dashed] (0.732,0.732) to (0,-2);
                \draw[-,cyan,very thick,dashed] (0.732,0.732) to (1.732*2,0);
            \end{scope}

            \begin{scope}[shift={(0,2)},rotate=300,shift={(0,2)}]
                \draw[-,cyan,very thick,dashed] (0.732,0.732) to (0,-2);
                \draw[-,cyan,very thick,dashed] (0.732,0.732) to (1.732*2,0);
            \end{scope}

            \draw[-,red!70!black,very thick,dashed] (0.732,0.732-0.03) to (2*1.732,-0.03);
            \draw[-,red!70!black,very thick,dashed] (0.732+0.03,0.732) to (0,-2+0.03);

            \begin{scope}[shift={(2,-2)},rotate=270]
                \draw[-,red!70!black,very thick,dashed] (0.732,0.732) to (2*1.732,0);
                \draw[-,red!70!black,very thick,dashed] (0.732,0.732) to (0,-2);
            \end{scope}

            \begin{scope}[shift={(2,-2)},rotate=270,shift={(-2,2*1.732)}]
                \draw[-,red!70!black,very thick,dashed] (0.732,0.732) to (2*1.732,0);
                \draw[-,red!70!black,very thick,dashed] (0.732,0.732) to (0,-2);
            \end{scope}

            \begin{scope}[shift={(2,-2*1.732)}]
                \draw[-,red!70!black,very thick,dashed] (0.732,0.732) to (2*1.732,0);
                \draw[-,red!70!black,very thick,dashed] (0.732,0.732) to (0,-2);
            \end{scope}

        \end{scope}
    \end{tikzpicture}
\end{minipage}

%% file: main.bib
@inproceedings{bridges23,
  author      = {Akpanya, Reymond and Goertzen, Tom and Wiesenhuetter, Sebastian and Niemeyer, Alice C. and Noennig, J\"{o}rg},
  title       = {Topological Interlocking, Truchet Tiles and Self-Assemblies: A Construction-Kit for Civil Engineering Design},
  pages       = {61--68},
  booktitle   = {Proceedings of Bridges 2023: Mathematics, Art, Music, Architecture, Culture},
  year        = {2023},
  url         = {http://archive.bridgesmathart.org/2023/bridges2023-61.html}
}

@article{InfluenceArrangement,
title = {Influence of block arrangement on mechanical performance in topological interlocking assemblies: A study of the versatile block},
journal = {International Journal of Solids and Structures},
volume = {306},
pages = {113102},
year = {2025},
doi = {https://doi.org/10.1016/j.ijsolstr.2024.113102},
author = {Tom Goertzen and Domen Macek and Lukas Schnelle and Meike Weiß and Stefanie Reese and Hagen Holthusen and Alice C. Niemeyer}
}

@inproceedings{TopologicalSymmetry,
  title={Topological Interlocking via Symmetry},
  author={Goertzen, Tom and Niemeyer, Alice C. and Plesken, Wilhelm},
  editor={Stokkeland, S. and Braarud, H.C},
  booktitle={6th fib International Congress on Concrete Innovation for Sustainability, 2022; Oslo; Norway},
  pages={1235--1244},
  year={2022},
}

@book{armstrong1997groups,
    AUTHOR = {Armstrong, M. A.},
     TITLE = {Groups and symmetry},
    SERIES = {Undergraduate Texts in Mathematics},
 PUBLISHER = {Springer-Verlag, New York},
      YEAR = {1988},
     PAGES = {xii+186},
MRREVIEWER = {Peter\ M.\ Neumann},
       DOI = {10.1007/978-1-4757-4034-9}
}

@PHDTHESIS{DissTom,
      author       = {Goertzen, Tom},
      othercontributors = {Niemeyer, Alice Catherine and Robertz, Daniel},
      title        = {{C}onstructing simplicial surfaces with given geometric
                      constraints},
      school       = {RWTH Aachen University},
      type         = {Dissertation},
      publisher    = {RWTH Aachen University},
      year         = {2024},
      url          = {https://publications.rwth-aachen.de/record/992149},
}

@misc{OEIS_A078099,
  author       = {N. J. A. Sloane},
  title        = {Entry {A}078099 in {T}he {O}n-{L}ine {E}ncyclopedia of {I}nteger {S}equences},
  howpublished = {\url{https://oeis.org/A078099}}
}

@book {TilingsAndPatterns,
    AUTHOR = {Gr\"unbaum, Branko and Shephard, G. C.},
     TITLE = {Tilings and patterns},
    SERIES = {A Series of Books in the Mathematical Sciences},
      NOTE = {An introduction},
 PUBLISHER = {W. H. Freeman and Company, New York},
      YEAR = {1989},
     PAGES = {xii+446},
      ISBN = {0-7167-1998-3},
   MRCLASS = {52A45},
  MRNUMBER = {992195},
}

@article{Miura-ori,
  title={Counting Miura-ori Foldings},
  author={Jessica Ginepro and Thomas C. Hull},
  journal={J. Integer Seq.},
  year={2014},
  volume={17},
  pages={14.10.8},
  url={https://api.semanticscholar.org/CorpusID:7879729}
}

@book{InternationalTablesA2002,
  address = {Dordrecht, Boston, London},
  author = {IUCr},
  edition = {5. revised edition},
  publisher = {Kluwer Academic Publishers},
  series = {International Tables for Crystallography},
  title = {International Tables for Crystallography, Volume A: Space Group Symmetry},
  year = 2002
}

@article{Tubular,
title = {Innovative approach for designing topological interlocking bricks with precise morphological representation and controlled interface curvature},
journal = {Materials \& Design},
volume = {253},
pages = {113844},
year = {2025},
doi = {https://doi.org/10.1016/j.matdes.2025.113844},
author = {Maliheh {Tavoosi Gazkoh} and Xiaoshan Lin and Annan Zhou}
}

@article{WEIZMANN201618,
title = {Topological interlocking in buildings: A case for the design and construction of floors},
journal = {Automation in Construction},
volume = {72},
pages = {18-25},
year = {2016},
doi = {https://doi.org/10.1016/j.autcon.2016.05.014},
author = {Michael Weizmann and Oded Amir and Yasha Jacob Grobman}
}

@book{abeille_memoire_1735,
	title = {Machines et inventions approuvées par l’Académie Royale des Sciences depuis son établissement jusqu’à present; avec leur Description},
	publisher = {Académie royale des Sciences},
	author = {Gallon, J.-G.},
	year = {1735}
}

@article{dyskin_new_2001,
	title = {A new concept in design of materials and structures: assemblies of interlocked tetrahedron-shaped elements},
	volume = {44},
	doi = {https://doi.org/10.1016/S1359-6462(01)00968-X},
	number = {12},
	journal = {Scripta Materialia},
	author = {Dyskin, A. V. and Estrin, Y. and Kanel-Belov, A. J. and Pasternak, E.},
	year = {2001},
	pages = {2689--2694},
}

@article{dyskin_toughening_2001,
	title = {Toughening by {Fragmentation}—{How} {Topology} {Helps}},
	volume = {3},
	doi = {https://doi.org/10.1002/1527-2648(200111)3:11<885::AID-ADEM885>3.0.CO;2-P},
	number = {11},
	journal = {Advanced Engineering Materials},
	author = {Dyskin, A. V. and Estrin, Y. and Kanel-Belov, A. J. and Pasternak, E.},
	year = {2001},
	pages = {885--888}
}

@article{Dyskin2003,
author = {Dyskin, A.V. and Estrin, Y. and Pasternak, E. and Khor, H.C. and Kanel-Belov, A.J.},
title = {Fracture Resistant Structures Based on Topological Interlocking with Non-planar Contacts},
journal = {Advanced Engineering Materials},
volume = {5},
number = {3},
pages = {116-119},
doi = {https://doi.org/10.1002/adem.200390016},
year = {2003}
}

@article{DelaunayLofts,
title = {Delaunay Lofts: A biologically inspired approach for modeling space filling modular structures},
journal = {Computers \& Graphics},
volume = {82},
pages = {73-83},
year = {2019},
issn = {0097-8493},
doi = {https://doi.org/10.1016/j.cag.2019.05.021},
author = {Sai Ganesh Subramanian and Mathew Eng and Vinayak R. Krishnamurthy and Ergun Akleman}
}

@article{VoroNoodles,
author = {Ebert, Matthew and Akleman, Ergun and Krishnamurthy, Vinayak and Kulagin, Roman and Estrin, Yuri},
title = {VoroNoodles: Topological Interlocking with Helical Layered 2-Honeycombs},
journal = {Advanced Engineering Materials},
volume = {26},
number = {4},
pages = {2300831},
doi = {https://doi.org/10.1002/adem.202300831},
year = {2024}
}

@book{Gorin2021,
  author    = {Vadim Gorin},
  title     = {Lectures on Random Lozenge Tilings},
  series    = {Cambridge Studies in Advanced Mathematics},
  publisher = {Cambridge University Press},
  year      = {2021}
}

@book{frezier1737,
  author    = {Fr{\'e}zier, Am{\'e}d{\'e}e-Fran{\c{c}}ois},
  title     = {La th{\'e}orie et la pratique de la coupe des pierres et des bois pour la construction des voutes et autres parties des b{\^a}timents civils},
  year      = {1737},
  address   = {Paris},
  publisher = {Jean-Daniel Doulsseker},
}
